\documentclass[11pt,a4paper,reqno]{amsart}
\usepackage[T1]{fontenc}
\usepackage{lmodern}
\usepackage[american]{babel}
\usepackage{csquotes}
\usepackage[final]{microtype}
\UseMicrotypeSet[protrusion]{basicmath}
\microtypecontext{spacing=nonfrench}

\usepackage{mathtools}
\usepackage{amssymb}
\usepackage{stmaryrd}
\usepackage{mathrsfs}
\usepackage{bm}
\mathtoolsset{centercolon}
\numberwithin{equation}{section}

\usepackage[margin=1in]{geometry}
\usepackage{enumitem}
\setlist{itemsep=1pt,topsep=1pt,parsep=1pt}
\setlist[itemize]{label=\textendash,leftmargin=1.5em,labelsep=0.5em}
\usepackage{tikz}
\usepackage{caption}
\usepackage[
  backend=biber,
  style=alphabetic,
  maxnames=6,
  minnames=3,
  giveninits=true,
  doi=false,
  url=false,
  isbn=false,
  backref=false
]{biblatex}
\renewbibmacro{in:}{
  \ifentrytype{article}{}{\printtext{\bibstring{in}\intitlepunct}}
}

\usepackage{aliascnt}
\theoremstyle{plain}
\newtheorem{theorem}{Theorem}[section]
\newcommand{\newaliastheorem}[2]{
  \newaliascnt{#1}{theorem}
  \newtheorem{#1}[#1]{#2}
  \aliascntresetthe{#1}
  \expandafter\def\csname #1autorefname\endcsname{#2}
}
\newaliastheorem{lemma}{Lemma}
\newaliastheorem{claim}{Claim}
\newaliastheorem{proposition}{Proposition}
\newaliastheorem{corollary}{Corollary}
\newaliastheorem{conjecture}{Conjecture}
\theoremstyle{definition}
\newaliastheorem{definition}{Definition}
\newaliastheorem{example}{Example}
\theoremstyle{remark}
\newaliastheorem{remark}{Remark}

\usepackage{hyperref}
\hypersetup{
  colorlinks,
  linkcolor={red!50!black},
  citecolor={blue!50!black},
  urlcolor={blue!80!black}
}
\usepackage[capitalize, nameinlink]{cleveref}
\crefname{claim}{claim}{claims}
\Crefname{claim}{Claim}{Claims}

\newif\ifnotes
\notesfalse
\ifnotes
  \newcommand{\note}[1]{\textbf{\textcolor{red}{[#1]}}}
\else
  \newcommand{\note}[1]{}
\fi

\DeclarePairedDelimiter\abs{\lvert}{\rvert}
\newcommand*{\indic}{\bm{1}}
\newcommand*{\Var}{\mathbb{V}\mathrm{ar}}
\newcommand*{\Cov}{\mathbb{C}\mathrm{ov}}
\newcommand*{\EE}{\mathbb{E}}
\newcommand*{\E}{\mathbf{E}}
\let\P\relax
\newcommand*{\P}{\mathbf{P}}
\newcommand*{\PP}{\mathbb{P}}
\newcommand*{\N}{\mathbb{N}}
\newcommand*{\Z}{\mathbb{Z}}
\newcommand*{\R}{\mathbb{R}}
\newcommand{\tPP}{\widetilde{\PP}}
\newcommand{\tEE}{\widetilde{\EE}}
\newcommand{\tVV}{\widetilde{\mathbb{V}}\mathrm{ar}}
\newcommand{\tC}{\widetilde{\mathbb{C}}\mathrm{ov}}
\newcommand{\Rhat}{\widehat{R}}
\newcommand*{\de}{\mathop{}\!\mathrm{d}}
\newcommand*{\dd}{\de}
\renewcommand*{\th}{\vartheta}
\newcommand*{\e}{\mathrm{e}}
\newcommand{\disc}{\mathrm{disc}}
\newcommand{\sumtwo}[2]{\sum_{\substack{#1 \\ #2}}}
\newcommand{\w}{\omega}
\newcommand*{\llb}{\llbracket}
\newcommand*{\rrb}{\rrbracket}
\newcommand{\cF}{\mathcal{F}}
\newcommand{\cI}{\mathcal{I}}
\newcommand{\cM}{\mathcal{M}}
\newcommand{\cU}{\mathcal{U}}
\newcommand{\cV}{\mathcal{V}}
\newcommand{\cY}{\mathcal{Y}}
\newcommand{\cZ}{\mathcal{Z}}

\title[Fractional moments of the SHF and directed polymers]{Fractional moments of the Stochastic Heat Flow and\\2D Directed Polymers}
\author[Q. Berger]{Quentin Berger}
\address{Université Sorbonne Paris Nord, Laboratoire d'Analyse, Géométrie et Applications, CNRS UMR 7539, 99 Av. J-B Clément, 93436 Villetaneuse, France and Institut Universitaire de France}
\email{quentin.berger@math.univ-paris13.fr}
\author[N. Turchi]{Nicola Turchi}
\address{Dipartimento di Matematica e Applicazioni, Università degli Studi di Milano-Bicocca, via Cozzi 55, 20125 Milano, Italy}
\email{nicola.turchi@unimib.it}
\author[N. Zygouras]{Nikos Zygouras}
\address{University of Warwick, Department of Mathematics, Coventry, CV47AL, UK}
\email{n.zygouras@warwick.ac.uk}
\subjclass[2020]{Primary: 82B44; Secondary: 60K35, 82D60.}
\keywords{Change of Measure, Coarse-Graining, Directed Polymer in Random Environment, Disordered Systems, Size Bias, Stochastic Heat Equation, Stochastic Heat Flow}

\begin{document}

\begin{abstract}
  \noindent
  We estimate the fractional moments of the normalized mass assigned by the Critical 2D Stochastic Heat Flow to small balls.
  Our results also cover the discrete case corresponding to the 2D directed polymer model and provide estimates that are uniform in all parameters.
  One key takeaway of our results is that the vanishing of the fractional moments is completely governed by the divergence of the second moment.
  We use a quite robust method, by refining the change of measure argument and introducing a novel coarse-graining procedure, reducing the proof to essentially second moment estimates (in fact, we also provide sharp second moment estimates for directed polymers, of independent interest).
\end{abstract}

\maketitle

\section{Introduction and Main Results}

The Critical 2D Stochastic Heat Flow (SHF) is a universal continuum model for a variety of physical phenomena related to disordered systems.
It was constructed in \cite{CSZ23} as a universal scaling limit of the directed polymers model under the critical disorder scaling in spatial dimension~\(2\).
It also serves as a non-trivial notion of solution to the two-dimensional Stochastic Heat Equation (SHE)
\begin{align}\label{SHE}
  \partial_t u
  = \frac{1}{2}\Delta u+\beta \xi u, \qquad t>0, \, x\in \R^2 \,,
\end{align}
with \(\xi\) denotes space-time Gaussian white noise in spatial dimension \(2\), see~\cite{Tsai24}.
We stress that the singularity of \(\xi\) renders \eqref{SHE} ill-posed and that~\eqref{SHE} falls outside the classical solution theory frameworks such as regularity structures or paracontrolled calculus.

There has been tremendous progress in the study of singular SPDEs (see \cite{Hai14, GIP15, K16, D25} for some of the main landmarks in the field) and the exploration of statistical mechanics at the critical dimension has had a substantial history (see \cite{ADC21, KW25} for some important recent developments and references).
The Critical 2D SHF, however, presents a rare example of a non-Gaussian random field at the critical dimension.
We refer to \cite{CSZ24, CSZ25-icm} for reviews and questions on the topic.

It was shown in \cite{CSZ25} that the time marginals of the SHF are singular with respect to the Lebesgue measure; note that~\cite{GN26} quantifies the singularity of the support.
In this article we probe into the fine-scale structure of the SHF as a two-dimensional measure, through estimates on the fractional moments of the mass that it assigns on small balls.
Such fractional estimates are relevant to understanding the multifractal properties of the SHF, as outlined in Section \ref{sec:multifractal} below.

Estimates on integer moments of the SHF have played an important role in the construction of the model \cite{CSZ19-3rd, GQT21, CSZ23} and have been used to demonstrate features of intermittency \cite{GN25, LZ26, GN26}.
Such integer moment estimates are facilitated by a mapping to the exponential moments of the collision time of two-dimensional random walks or Brownian motions, as well as links to the ground states of the delta-Bose gas \cite{DFT94, DR04, GQT21}.
These links allow for the use of a variety of tools, such as operator theory \cite{DFT94, AGHKH05, GQT21, CSZ23}, Laplace methods and recursion analysis \cite{CSZ19-3rd, CZ23, LZ26} or the theory of the Gaussian Free Field \cite{GN25}.
All these powerful tools, however, become unavailable when dealing with fractional moments, as the key relation to the collisions of random walks, Brownian motions or to the delta-Bose gas do not extend to noninteger moments.

To proceed with our task of estimating fractional moments, we employ ideas that have been introduced in the context of disordered systems \cite{DGLT09} and refined in studies of the free energy of the directed polymer model \cite{Lac10a, BL17, JL24a}, which combine a change of measure argument with a coarse-graining scheme.
This method allows for sharp upper bounds on the free energy for the 2D directed polymer and sharp estimates on fractional moments of the mass assigned to balls for the SHF in a strong disorder regime, see \cite{BCT25} for details.
However, so far, the method  does not yield any non-trivial estimate of the mass assigned to \textit{balls with a vanishing radius}.
We introduce here two novelties:
\begin{enumerate}
  \item We develop a conceptually general change of measure argument (inspired but different from that of~\cite{BCT25}), that allows us to obtain a sub-optimal estimate on fractional moments for balls with a vanishing radius;
  \item We bootstrap these sub-optimal estimates through a \textit{nonhomogeneous} coarse-graining scheme, which takes into account the separation of scales of the model.
\end{enumerate}
In practice, we develop this approach in the context of the two-dimensional directed polymer model and we transfer the fractional-moment estimates obtained in this setting to the Critical 2D SHF, via a limiting procedure.
We provide a more detailed sketch of this general strategy in \Cref{sec:strategy}.
We also stress that a very recent article~\cite{Huang26} also investigates the fractional moments of the Critical 2D SHF (in the rest of the text we will abbreviate to just SHF): we comment on this result in \Cref{sec:lower-bound} below.

\subsection{Main results for the SHF: upper bound on the fractional moments}
\label{sec:SHFresults}

The SHF is a continuous-time, measure-valued process, which we will denote by \( \mathscr{Z}_{s,t}^{\th}(\dd x, \dd y) \) for \(0<s<t\).
Here, \(\dd x, \dd y\) are volume elements in \(\R^2\) indicating the starting and ending points of the flow at times \(s, t\), respectively.
The parameter \(\th\) plays an important role as it allows for the SHF to interpolate between the so-called \textit{weak disorder} (or \textit{subcritical}) regime and the \textit{strong disorder} (or \textit{supercritical}) regime.
Details on the emergence and the precise definition of the parameter \(\th\) will be given in \Cref{sec:DPREresults}, see in particular \eqref{choiceb}-\eqref{def:theta-N-beta}.

We will focus here on the marginal
\[
  \mathscr{Z}_{t}^{\th}(\dd x)
  \coloneqq \mathscr{Z}_{0,t}^{\th}(\dd x, \R^2)  \,,
\]
and we will use the notation
\[
  \mathscr{Z}_{t}^{\th}( \varphi)
  \coloneqq \int_{\R^2}  \varphi(x) \, \mathscr{Z}_{t}^{\th}(\dd x) \,.
\]
Notice that, if \(\varphi\) is a probability density on \(\R^2\), then we have \(\EE[\mathscr{Z}_{t}^{\th}( \varphi) ]=1\).
We will be particularly interested in the case where \( \varphi\) is the uniform distribution in a ball of radius \(r\), so we also introduce the notation:
\[
  \mathcal{U}_r(x)
  \coloneqq \frac{1}{\pi r^2} \, \indic_{B(0,r)}(x) \qquad \text{with} \qquad B(0,r)
  \coloneqq \{x\in\R^2 \colon \abs{x}
  \le r\} \,.
\]

Our first main result is the following upper bound on fractional moments.
One important point that we emphasize here is that the result is valid \textit{uniformly} in the parameters \(t,\th,r\), with \(r\leq \sqrt{t}\).
In particular it holds for balls with vanishing radius \(r\downarrow 0\) and/or diverging time \(t \to\infty\), as long as the second moment diverges (it is otherwise trivial).

\begin{theorem}[Averaged mass of the SHF on balls, upper bound]
  \label{thm:balls-SHF}
  There exist constants \(C,c>0\) such that, for any \(p \in (0,1)\), uniformly in \(t>0\), \(\th\in \R\) and \(r \in (0,\sqrt{t}]\), we have
  \begin{equation}
    \label{eq:balls-SHF}
    \EE\bigl[\mathscr{Z}_{t}^{\th}(\cU_r)^{p}\bigr]
    \leq C\, \EE\bigl[\mathscr{Z}_{t}^{\th}(\cU_r)^{2}\bigr]^{c\, p (p-1)} \,.
  \end{equation}
  In fact, we  have a more explicit upper bound: there exist constants \(C',c'>0\) such that, for any \(p \in (0,1)\), uniformly in \(t>0\), \(\th\in \R\) and \(r \in (0,\sqrt{t}]\), we have
  \begin{equation}
    \label{eq:balls-SHF-explicit}
    \EE\bigl[\mathscr{Z}_{t}^{\th}(\cU_r)^{p}\bigr]
    \leq C' \Bigl(1+ \log\Bigl(1+\frac{t}{r^2}\Bigr)\,\cV\bigl(t \,\e^{\th}\bigr) \Bigr)^{c'\, p (p-1)}\,,
  \end{equation}
  where \(\cV(T) \coloneqq \int_0^{\infty} \frac{T^s}{\Gamma(s+1)} \dd s\) is the so-called Volterra function (see~\eqref{eq:asymp-V} for its asymptotics as \(T\) goes to \(0\) or~\(\infty\)).
\end{theorem}

\Cref{thm:balls-SHF} is a signature of intermittency as~\eqref{eq:balls-SHF} shows that whenever we are in a regime where the second moment of \(\mathscr{Z}_{t}^{\th}(\cU_r)\) blows up, then its fractional moments \(p\in (0,1)\) go to zero, which also implies that \(\mathscr{Z}_{t}^{\th}(\cU_r)\) converges to \(0\) in probability.
This is consistent with the result of both \cite{CSZ25}, which shows the singularity of the SHF with respect to the Lebesgue measure, and the more recent result \cite{GuTsai26}, which shows that the Lebesgue volume of the set \(\{ x\in D \colon  \mathscr{Z}_{t}^{\th}(\cU_r(x)) \notin(\abs{\log r}^{-\frac{1}{2}-\varepsilon}, \abs{\log r}^{-\frac{1}{2}+\varepsilon}) \}\), with \(D\subset\R^2\) bounded, tends to \(0\) in~\(L^1(\PP)\) as~\(r\downarrow 0\).

We also stress that a corresponding lower bound on fractional moments may be achieved, see \Cref{sec:lower-bound} below, which highlights the fact that the vanishing of the fractional moment is in fact completely governed by the divergence of the second moment.

Let us conclude here by mentioning that going from~\eqref{eq:balls-SHF} to \eqref{eq:balls-SHF-explicit} (or vice-versa) is a simple matter of estimating the second moment.
The subtlety is that the statement of \Cref{thm:balls-SHF} is \textit{uniform in the parameters \(t>0\), \(\th\in \R\), \(r\in (0,\sqrt{t}]\)}, so one needs a second moment estimate which holds uniformly in the parameters, as follows.

\begin{proposition}[Variance of the SHF]
  \label{prop:var-SHF}
  There are constants \(c_1,c_2\), \(c_3,c_4\) such that, for any \(t>0\), \(\th \in \R\) and  \(r>0\),
  \[
    c_1  \log \Bigl(1+\frac{t}{r^2}\Bigr) \, \mathcal{V}\bigl( c_2 t\,\e^{\th} \bigr)
    \leq \Var\bigl(\mathscr{Z}_{t}^{\th}(\cU_r)\bigr)
    \leq c_3  \log \Bigl(1+\frac{t}{r^2}\Bigr) \, \mathcal{V}\bigl( c_4 t\,\e^{\th} \bigr) \,.
  \]
\end{proposition}

Let us mention the following asymptotics for the Volterra function:
\begin{equation}
  \label{eq:asymp-V}
  \mathcal{V}(T) \sim  \e^T \quad \text{as } T \to+\infty \,, \qquad \mathcal{V}(T) \sim \frac{1}{\log(1/T)} \quad \text{as } T\to 0^+ \,.
\end{equation}
In particular, the second moment \(\EE[\mathscr{Z}_{t}^{\th}(\cU_r)^{2}]\) diverges (or equivalently \(\Var(\mathscr{Z}_{t}^{\th}(\cU_r))\) diverges) for two possibly different reasons (say in the regime \(r\in (0,\sqrt{t}]\)): either \(t\e^{\th} \to \infty\) or \(r/\sqrt{t} \to 0\) (sufficiently fast if \(t\e^{\th} \downarrow 0\)).

\subsection{A non-uniform result and a conjecture}
\label{sec:lower-bound}

A very recent result by Gu and Tsai~\cite{GuTsai26} (see also \cite{CD24, CSZ25} for results in the same spirit but in the `subcritical' and `quasicritical' regimes) shows that, for any fixed \(t>0\) and \(\th \in \R\), there is some deterministic constant \(\alpha_r = \alpha_{r}(t,\th)\) with \(\lim_{r\downarrow 0} \alpha_r =0\) such that
\begin{equation}
  \label{eq:log-normality}
  \frac{1}{\sqrt{\log \log \frac1r}}\Bigl(\log \mathscr{Z}_{t}^{\th}(\cU_r) + \frac{1+\alpha_r}{2} \log \log \frac1r\Bigr) \xrightarrow[\;r\to 0\;]{(d)} \mathcal{N}(0,1) \,,
\end{equation}
where \(\mathcal{N}(0,1)\) is a standard normal random variable.
This gives some important information on the mass assigned to small balls by the SHF, but one cannot deduce our bounds~\eqref{eq:balls-SHF}-\eqref{eq:balls-SHF-explicit} on the fractional moments, since~\eqref{eq:log-normality} does not yield information on the right tail \(\log \mathscr{Z}_{t}^{\th}(\cU_r)\) far from the normal regime covered by~\eqref{eq:log-normality}.

During completion of this paper, it was announced in~\cite{Huang26} that for all \(p\in (0,1)\) and all \(t,\vartheta\), it holds
\begin{equation}
  \label{eq:sharp}
    \EE\bigl[\mathscr{Z}_{t}^{\th}(\cU_r)^{p}\bigr]
    = \log \Bigl(\frac{1}{r}\Bigr)^{\frac{p(p-1)}{2} +o(1)} \qquad \text{ as } r\downarrow 0\,,
\end{equation}
where the dependence on \(t,\th\) is hidden in the \(o(1)\).
Let us also stress that this result was first obtained for integer moments in \cite{LZ26}.

The result~\eqref{eq:sharp} is in some sense sharper than our \Cref{thm:balls-SHF}, because it provides the correct behavior for \(\log \EE[\mathscr{Z}_{t}^{\th}(\cU_r)^{p}]\) in the limit \(r\downarrow 0\), and in particular identifies the correct exponent \(\frac12 p(p-1)\).
Note however that, contrary to our \Cref{thm:balls-SHF}, the parameters \(t,\th\) in \eqref{eq:sharp} are required to be fixed and one only sends \(r\) to \(0\).
In other words, our result is \emph{uniform} in the parameters \(t,\th\) in the small ball regime \(r\leq \sqrt{t}\).
We also believe that an advantage of our approach is that, even though it does not yield the sharpest result, it provides general tools to deal with fractional moments of disordered systems and could be useful in a broader context.

Let us now formulate the conjecture that the exponent \(\frac{1}{2}p(1-p)\) is in fact valid for all \(p\in \R\) and holds uniformly in the parameters \(t,\vartheta\), \(r\leq \sqrt{t}\).

\begin{conjecture}
  \label{1/2conj}
  As soon as \(\EE\bigl[ \mathscr{Z}_{t}^{\th}(\cU_r)^2\bigr] \to\infty\)  with \(r\leq \sqrt{t}\), then for any \(p\in \R\) we have that
  \begin{equation}
    \label{eq:conjecture}
    \EE\bigl[\mathscr{Z}_{t}^{\th}(\cU_r)^{p}\bigr]
    = \EE\bigl[ \mathscr{Z}_{t}^{\th}(\cU_r)^2\bigr]^{\frac{p(p-1)}{2} +o(1)}  \,,
  \end{equation}
  where \(o(1)\) is a quantity that depends on \(t,\th,r,p\) and goes to \(0\) as soon as \(\EE\bigl[ \mathscr{Z}_{t}^{\th}(\cU_r)^2\bigr] \to\infty\).
\end{conjecture}

This conjecture thus suggests that the divergence of the second moment of the SHF determines the asymptotic behavior of all \(p\)-th moments.
As mentioned above, this conjecture is proven for integer moments \(p\in \N\) in \cite{LZ26}. In the case of fractional moments \(p\in (0,1)\), it is supported by \Cref{thm:balls-SHF} and \eqref{eq:sharp}.
On the other hand, negative moments have so far been elusive.
Even though some first results on negative tails of the SHF have appeared in~\cite{N25mart}, these are still not enough to conclude information on negative moments.
Our belief that \eqref{eq:conjecture} is still valid for negative moments stems from the approximate log-normality of the SHF in small scales, see~\eqref{eq:log-normality}.
Notice that the relation \(\EE[Z^p] = \EE[Z^{2}]^{\frac{p(p-1)}{2}}\) holds for mean-one log-normal distributions.

\subsection{Main results for directed polymers}
\label{sec:DPREresults}

In this subsection we introduce the directed polymer model and its relation to the SHF.
We then present our main theorems in this context: their proofs constitute the bulk of the paper, and the corresponding results on the SHF then follow by taking a scaling limit of the polymer partition functions.

We consider a family of i.i.d.\ random variables \((\omega(n,x))_{n\in \N, x\in \Z^2}\) with mean \(0\), variance \(1\) and finite exponential moments \(\lambda(\beta)\coloneqq \log \EE\bigl[ e^{\beta \omega(n,x)} \bigr]<\infty\) for \(\beta\in[-3,3]\).
We also consider a simple, symmetric random walk \(S=(S_n)_{n\geq 0}\) on \(\Z^2\) and we denote its distribution and expectation when starting from initial position \(S_0=x\) by \(\P_x \) and \(\E_x\), respectively.

The point-to-plane partition function for a random walk (the `polymer') \(S\) starting at \(x\in \Z^2\) is defined by
\begin{equation}
  \label{eq:polymeas}
  Z_{N}^{\beta,\w}(x)
  \coloneqq \E_x\biggl[\exp\Bigl(\sum_{n=1}^{N} (\beta\w(n,S_n)-\lambda(\beta))\Bigr)  \biggr]\,,
\end{equation}
and when \(x=0\) we simply write \(Z_{N}^{\beta,\w}\coloneqq Z_{N}^{\beta,\w}(0)\).
We will also consider the case when the random walk starts from an initial distribution \(f\) and write
\[
  Z_{N}^{\beta,\w}(f)
  \coloneqq \sum_{x\in \mathbb{Z}^2} f(x)\, Z_{N}^{\beta,\w}(x) \,.
\]

\subsubsection*{The Critical 2D SHF}

In \cite{CSZ23}, the authors construct the unique limit of the field of partition functions\footnote{In fact, \cite{CSZ23} construct the unique limit of the field of point-to-point partition functions as measures on \(\R^2\times \R^2\), but we only consider here the first marginal.}
\begin{align}\label{preSHF}
  \cZ_{N; \,t}^{\beta_N}(\dd x)
  \coloneqq Z _{[ Nt ]}^{\beta_N,\omega} \bigl( \llbracket \sqrt{N}x \rrbracket \bigr) \dd x\, , \qquad 0
  \le t<\infty \,,
\end{align}
where \([\cdot]\) maps a real number to its nearest, even integer neighbor, and \(\llbracket\cdot\rrbracket\) maps points of~\(\R^2\) to their nearest, even integer point on \(\Z^2_\text{even}\coloneqq \{ (z_1,z_2)\in\Z^2:z_1+z_2\in 2\Z \}\), and \(\dd x\) is the Lebesgue measure on \(\R^2\).

To explain how the temperature parameter \(\beta_N\) is (critically) tuned in \cite{CSZ23}, let us introduce
\[
  \sigma^2(\beta)
  \coloneqq  \e^{\lambda(2\beta)-2\lambda(\beta)}-1 \qquad \text{ and } \qquad R_N
  \coloneqq \sum_{n=1}^{N} \P(S_{2n}
  =0) \,,
\]
and let us recall the asymptotics for the local time \(R_N\), see \cite{CSZ19-Dickman}:
\begin{align}\label{RN-asym}
  \pi R_N
  = \log N + \alpha +o(1), \quad \text{with} \quad \alpha
  \coloneqq 4\log 2 +\gamma -\pi,
\end{align}
where \(\gamma\) is the Euler--Mascheroni constant.
Then, we let \((\beta_N)_{N\geq 1}\) be tuned in the following critical window:
\begin{align}\label{choiceb}
  \sigma^2(\beta_N)
  = \frac{1}{R_N} \Bigl(1+\frac{\th+o(1)}{\log N} \Bigr),
\end{align}
where \(\th\in \R\) is a disorder strength / interpolation parameter and \(o(1)\) denotes asymptotically negligible corrections as \(N\to \infty\) (see also \cite{BC98} for the origins of this critical scaling in the  case of SHE).
We can now state (a consequence of) the main result of \cite{CSZ23}.

\begin{theorem}[\cite{CSZ23}]
\label{thm:CSZ}
  Let \((\beta_N)_{N\geq 1}\) be as in~\eqref{choiceb} for some \(\th\in\R\) and let \(\bigl( \cZ_{N; \,t}^{\beta_N}(\dd x) \bigr)_{0\le t<\infty}\) be defined as in \eqref{preSHF}.
  Then, as \(N\rightarrow\infty\), the process of random measures \((\cZ_{N; t}^{\beta_N}(\dd x))_{0\le t<\infty}\) converges to a unique limit \((\mathscr{Z}_{t}^{\th}(\dd x))_{0\le t<\infty}\), which is a marginal of the \emph{Critical 2D Stochastic Heat Flow}.
  In particular,
  \[
    \forall \, \varphi\in \mathcal{C}_c^{\infty}(\R^2), \qquad \int_{\R^2} \varphi(x) \cZ_{N; \,t}^{\beta_N}(\dd x)
    \xrightarrow[\;N\to\infty]{(d)} \int_{\R^2} \varphi(x) \mathscr{Z}_{t}^{\th}(\dd x)
    \eqqcolon \mathscr{Z}_{t}^{\th}(\varphi) \,.
  \]
\end{theorem}

A different point of view on the critical scaling window \eqref{choiceb} was introduced in \cite{BCT25}.
More precisely, given any \(\beta > 0\) and \(N\in\N\), the interpolating parameter (which determines the `disorder strength') is defined as:
\begin{equation}
  \label{def:theta-N-beta}
  \th(N,\beta)
  \coloneqq \pi R_N - \frac{\pi}{\sigma^2(\beta)} \,.
\end{equation}
In particular, if we are in a regime where \(N\to\infty\), \(\beta \downarrow 0\) in such a way that \(\th(N,\beta) \to \th \in \R\), we then recover \eqref{choiceb}.
The advantage of definition~\eqref{def:theta-N-beta} is that it also works \emph{outside} of the critical window.
In particular, it allows us to formulate results that hold uniformly in the parameters \(\beta\in (0,1)\) and \(N\in \N\), only in terms of \(\th(N,\beta)\); we mostly think about \(N\to\infty\) and \(\beta\downarrow 0\).
In fact the subcritical, critical and supercritical regimes can be defined as follows: we are in the \emph{subcritical regime} if \(\th(N,\beta) \to-\infty\), in the \emph{critical regime} if \(\th(N,\beta) \to \th\in \R\) and in the \emph{supercritical regime} if \(\th(N,\beta) \to+\infty\).
We adopt the formulation \eqref{def:theta-N-beta} in this paper, as well.

\subsubsection*{Main results for directed polymers}

We now state our result in the polymer setting, which captures the asymptotics of the fractional moments of the polymer partition starting from small balls.
We first denote by \(U_R\) the uniform distribution on the (discrete) ball of radius \(R\): for any \(R\geq 0\), let
\[
  U_R(x)
  \coloneqq \frac{1}{\abs{B(0,R) \cap \Z_{\mathrm{even}}^2}} \indic_{B(0,R)} (x) \qquad \text{ for } x\in \Z_{\mathrm{even}}^2 \,.
\]

Note that \(U_0 = \delta_0\) is the Dirac mass (Kronecker delta) at \(0\), so our result also applies to the point-to-plane partition function \(Z_{N}^{\beta,\w}=Z_{N}^{\beta,\w}(U_0)\).

\begin{theorem}[Polymers starting uniformly from a small ball]
  \label{thm:balls-polymers}
  There are constants \(c,C>0\) such that, for all \(p\in (0,1)\) all \(N\in \N\), 
  all \(\beta\in (0,1)\), and all \(0\leq R \leq \sqrt{N}\), we have
  \begin{equation}
    \label{eq:balls-polymers}
    \EE\bigl[ Z_{N}^{\beta,\w}(U_R)^p \bigr]
    \leq C\, \EE\bigl[ Z_{N}^{\beta,\w}(U_R)^2 \bigr]^{c\, p(p-1)} \,  .
  \end{equation}
  More explicitly, let \(\th(N,\beta)\) be defined as in \eqref{def:theta-N-beta}.
  Then there are constants \(c',C'\) such that, for all \(p\in (0,1)\), all \(N\in \N\), all \(\beta\in (0,1)\), and all \(0\leq R \leq \sqrt{N}\), we have
  \begin{equation}
    \label{eq:balls-polymers-explicit}
    \EE\bigl[ Z_{N}^{\beta,\w}(U_R)^p \bigr]
    \leq C'\, \Bigl(1+\log \Bigl(1+\frac{N}{1+R^2}\Bigr)\, \cV\bigl(\e^{\th(N,\beta)}\bigr) \Bigr)^{c'\, p(p-1)} \,,
  \end{equation}
  where \(\cV(T) \coloneqq \int_0^{\infty} \frac{T^s}{\Gamma(s+1)} \dd s\) is the Volterra function.
\end{theorem}

Let us stress here that our result is not limited to the critical regime where \(\th(N,\beta) \to \th\in \R\), and we in fact allow for ranges of parameters with \(\th(N,\beta) \to \pm \infty\).
Again, one important feature of \Cref{thm:balls-polymers} is that the result is uniform in the parameters \(N\in \N\), \(\beta\in (0,1)\) and \(R\leq \sqrt{N}\).
It shows that the fractional moment vanishes as soon as the second moment diverges, \textit{i.e.}\ as soon as \(\log (\frac{N}{1+R^2}) \mathcal{V}(\e^{\th(N,\beta)})\) diverges (in the case \(R\leq \sqrt{N}\)).

Let us finally mention that going from~\eqref{eq:balls-polymers} to~\eqref{eq:balls-polymers-explicit} (or vice-versa) is a matter of estimating (sharply) the second moment.
Once again, the difficulty is to provide estimates that hold \textit{uniformly} in the parameters \(N\in \N\), \(\beta\in (0,1)\), \textit{i.e.}\ are not limited to the critical regime.
This does not seem to appear in the literature, so we provide the following result here.

\begin{proposition}[Variance of polymer partition functions]
  \label{prop:var-polymer}
  Let \(\th(N,\beta)\) be defined as in \eqref{def:theta-N-beta}.
  Then there are constants \(c_1,c_2,c_3,c_4\), such that for any \(\beta\in (0,1)\), any \(N\in \N\) and any \(R\geq 0\), we have
  \[
    c_1 \log\Bigl( 1+ \frac{N}{1+R^2}\Bigr) \, \cV\bigl( c_2\, \e^{\th(N,\beta)} \bigr)
    \leq \Var\bigl[ Z_{N}^{\beta,\w}(U_R) \bigr]
    \leq c_3 \log\Bigl( 1+ \frac{N}{1+R^2}\Bigr) \, \cV\bigl( c_4\, \e^{\th(N,\beta)} \bigr) \,,
  \]
  Note that \(\e^{\th(N,\beta)} = (1+o(1)) \e^{\alpha}\,  N \e^{ -\frac{\pi}{\sigma^2(\beta)}}\) as \(N\to\infty\) (uniformly in \(\beta\in(0,1)\)), recalling~\eqref{RN-asym}.
\end{proposition}

We stress that the result is valid for all \(R\geq 0\), but we will only use it for \(R\leq \sqrt{N}\); in the case \(R>\sqrt{N}\), we may replace \(\log( 1+ \frac{N}{1+R^2})\) by \(\frac{N}{1+R^2}\).

\subsection{Towards a multifractal formalism for SHF}
\label{sec:multifractal}

It is expected that the SHF is a multifractal measure, see for example the simulations in \cite{CSZ25-icm, CSZ24} and the very recent~\cite{GN26} where authors show some logarithmic intermittency of the SHF.
Estimates on moments like the ones we obtain in this article are a step towards such framework.
Let us informally motivate why.

The notion of \(L^p\) (multi-fractal) spectrum for a measure $\mu$ was introduced by Parisi-Frisch \cite{PF85}, see also \cite{Ols95}.
This notion is captured by an exponent \(\tau_\mu(p)\) which (informally) describes the asymptotic \(\sup \sum_i \mu(B_r(x_i))^p \approx r^{\tau_\mu(p)}\) as \(r\downarrow 0\), where the supremum is over all countable families of disjoint balls of radius \(r\) and centres \(x_i\).
Overlooking the supremum, the expected value of such a sum for the SHF is \(\EE \bigl[\sum_i \mathscr{Z}_{t}^{\th} (B_r(x_i))^p \bigr] =  \sum_i  \EE \bigl[ \mathscr{Z}_{t}^{\th}(B_r(x_i))^p \bigr] \approx  r^{2p-2} \bigl(\log \frac{1}{r}\bigr)^{\frac{1}{2}p(p-1) +o(1)}\), where the term \(r^{2p}\) is just the \(p\)-th power of the Lebesgue volume of \(B_r(x)\) and \(r^{-2}\) is the number of such balls that fit inside a unit square, while the logarithmic correction to the Euclidean counting comes from the results of this work with an exponent \(cp(p-1)\), while \eqref{eq:sharp} would imply \(c=1/2\).

\subsection{Overview of the rest of the article}

Let us now give a brief outline of how the rest of the paper is organized:
\begin{itemize}
  \item In \Cref{sec:strategy}, we give some details on the overall strategy of the proof: the two main steps consist of a (key) sub-optimal bound on the fractional moment (\Cref{prop:key}) followed by a bootstrapping of a non-homogeneous coarse-graining procedure (\Cref{prop:coarse}).
        At the end of \Cref{sec:key} we also derive the proofs of \Cref{thm:balls-polymers,thm:balls-SHF} as a direct consequence of the key estimates therein.

  \item In \Cref{sec:keyprop}, we prove \Cref{prop:key} via a streamlined change of measure argument.
        The general idea is presented in \Cref{sec:change-measure,sec:scales-moment} and the technical estimates are proved in the subsequent sections.
  \item In \Cref{sec:coarse}, we prove \Cref{prop:coarse}.
        The proof follows the lines of the well-established coarse-graining procedure, which is
         adapted here to allow for non-homogeneous scales in the coarse-graining procedure.
  \item Finally, we collect all key second moment estimates in \Cref{sec:second-moment}.
        In particular, we prove \Cref{prop:var-SHF,prop:var-polymer} and some further technical estimates needed in \Cref{sec:keyprop}.
\end{itemize}

\section{Main steps of the proof: bootstrapping and coarse-graining}
\label{sec:strategy}

\subsection{Key estimates}
\label{sec:key}
The core of the paper consists in proving \Cref{thm:balls-polymers}.
First of all, let us explain how we may reduce to the following statement, which estimates the half moment \(p=\frac12\).
Let us define here 
\begin{equation}
  \label{def:M1-disc}
  \mathcal{M}_1^{\mathrm{disc}}(r) \coloneqq \Bigl\{ f: \Z^2 \to [0,1] \;;\; \sum_{z\in \Z^2} f(z) =1 \text{ and } f(z) =0 \text{ for } |z|_{\infty}>r\Bigr\} \,.
\end{equation}
\begin{theorem}
  \label{thm:half-moment}
  Let \(\th(N,\beta)\) be defined as in \eqref{def:theta-N-beta}.
  There exist constants \(C,c>0\) such that, for every \(\beta\in(0,1)\), \(N\in\N\), and \(0\le R\le\sqrt N\) we have
  \[
    \sup_{f\in\cM_1^{\rm disc}(R)} \EE\bigl[Z_N^{\beta,\w}(f)^{1/2}\bigr]
    \leq C \biggl( 1+ \log\Bigl(1+\frac{N}{1+R^2}\Bigr) \cV\bigl(\e^{\th(N,\beta)}\bigr) \biggr)^{-c}.
  \]
  Note that, together with \Cref{prop:var-polymer} and~\eqref{eq:asymp-V}, this gives, for some \(c'>0\)
  \[
    \sup_{f\in\cM_1^{\rm disc}(R)} \EE\bigl[Z_N^{\beta,\w}(f)^{1/2}\bigr]
    \leq C\, \EE\bigl[Z_N^{\beta,\w}(U_R)^2\bigr]^{-c}.
  \]
\end{theorem}

We can then turn estimates on the half-moment \(\EE[Z_N^{\beta,\w}(f)^{1/2}]\) into estimates for all \(p\in(0,1)\), thanks to the following lemma.

\begin{lemma}
  \label{lem:frac-moment}
  Let \(Z\geq0\) satisfy \(\EE[Z]=1\).
  Then, for every \(p\in(0,1)\),
  \begin{equation}
    \label{eq:half-moment-interpolation}
    \EE[Z^p]
    \leq \bigl(\EE[Z^{1/2}]^{2} \bigr)^{p\wedge (1-p)}
    \leq \bigl(\EE[Z^{1/2}]^{2} \bigr)^{p (1-p)}.
  \end{equation}
\end{lemma}
\begin{proof}
  If \(0<p\leq\frac12\), the map \(x\mapsto x^{2p}\) is concave.
  Therefore, Jensen's inequality gives
  \[
    \EE[Z^p]
    = \EE\bigl[(Z^{1/2})^{2p}\bigr]
    \leq \EE[Z^{1/2}]^{2p}.
  \]
  If \(\frac12\leq p<1\), write \(p=2(1-p)\frac12+(2p-1)\cdot1\).
  By Hölder's inequality,
  \[
    \EE[Z^p]
    \leq \EE[Z^{1/2}]^{2(1-p)} \EE[Z]^{2p-1}
    = \EE[Z^{1/2}]^{2(1-p)}.
  \]
  Combining the two cases proves the first inequality in \eqref{eq:half-moment-interpolation}.
  The second one holds simply because \(p \wedge (1-p) \ge p(1-p)\) and \(\EE[Z^{1/2}]^{2} \le \EE[Z] = 1\).
\end{proof}

\begin{proof}[Proof of \Cref{thm:balls-polymers}]
The bounds in \Cref{thm:balls-polymers}
 follow directly from the combination of \Cref{thm:half-moment,lem:frac-moment}.
\end{proof}

\begin{proof}[Proof of \Cref{thm:balls-SHF}]
The proof follows from Theorem \Cref{thm:balls-polymers} by taking the limit $N\to\infty$ and using
Theorem  \Cref{thm:CSZ}.
\end{proof}

\subsection{Key steps of the proof of \texorpdfstring{\Cref{thm:half-moment}}{}}

\subsubsection*{Step 1: A sub-optimal estimate}

We will first obtain a \textit{sub-optimal} bound, which is our key result.
The estimate is sub-optimal in two aspects: first, it applies only to parameters \(\th(M,\beta)\leq 0\); second, its decay is much weaker than in \Cref{thm:half-moment}.
However, its important feature is that it is \emph{quantitative} and \emph{uniform} in the parameters.
This will be enough in the coarse-graining procedure, since it will give a fixed contraction of the half moment on every coarse-graining block, which can then be iterated.

\begin{proposition}[Key proposition]
  \label{prop:key}
  There is a (universal) constant \(C_0\) such that the following holds: if \(\beta\in(0,1)\), \(M\in\N\), and \(0\leq R\leq\sqrt M\) are such that \(\th(M,\beta)\leq0 \), then
  \[
    \sup_{f\in\cM_1^{\rm disc}(R)} \EE\Bigl[Z_M^{\beta,\w}(f)^{1/2}\Bigr]
    \leq C_0\Biggl(1+\log\Biggl(1+\frac{\log\bigl(1+\frac{M}{1+R^2}\bigr)}{1+\abs{\th(M,\beta)}}\Biggr)\Biggr)^{-1/2}\,.
  \]
\end{proposition}

The proof of \Cref{prop:key} is carried out in \Cref{sec:keyprop} and follows a change of measure argument inspired by~\cite{BCT25} but which needs to be adapted here.
Note that \Cref{prop:key} already shows that the half-moment goes to \(0\) if the second moment diverges (in the case \(\th(M,\beta)\leq 0\)), recalling \Cref{prop:var-polymer} and \eqref{eq:asymp-V}.

\subsubsection*{Step 2: Scale-tailored coarse-graining}

Second, we use a coarse-graining procedure to obtain the correct decay for the fractional moment.
We will prove the following.

\begin{proposition}[Coarse-graining]
  \label{prop:coarse}
  For \(\beta \in (0,1)\), \(N \in \N\) and \(0\leq R \leq \sqrt{N}\), let us introduce some intermediate scales \(M_0\coloneqq \lfloor 1+R^2\rfloor <M_1 < \cdots< M_{\ell} \leq N\).
  If we have that
  \begin{equation}
    \label{eq:all-scales}
    \forall 1
    \leq i
    \leq \ell \qquad \sup_{f \in \cM_1^{\disc}(\sqrt{M_{i-1}})} \EE\bigl[Z_{M_i-M_{i-1}}^{\beta,\w}(f)^{1/2}\bigr]
    \le \frac{1}{300} \,,
  \end{equation}
  then we get that
  \begin{equation}
    \label{eq:all-scales2}
    \sup_{f \in \cM_1^{\disc}(R)} \EE\bigl[Z_{N}^{\beta,\w}(f)^{1/2}\bigr]
    \le 20\, \e^{- \ell} \sup_{f \in \cM_1^{\disc}(\sqrt{M_{\ell}})} \EE\bigl[Z_{N-M_{\ell}}^{\beta,\w}(f)^{1/2}\bigr]
    \leq 20\, \e^{- \ell} \,.
  \end{equation}
\end{proposition}

The proof is given in \Cref{sec:coarse}, and follows the lines of \cite[Prop.~3.4]{BCT25}, with some important novelty here: the coarse-graining scales are not constant and are in fact growing, namely we will use scales with \(M_i\gg M_{i-1}\) in \Cref{prop:coarse}.
This is one important novelty of the present article, since coarse-graining procedures (to our knowledge) had only been used so far with constant scales \(M_k= k M\) for some fixed \(M\).
We then apply the previous sub-optimal upper bound on each of the coarse-graining scales.
Crucially, the estimate of \Cref{prop:key} is \emph{quantitative}, with a constant \(C_0\) uniform in the parameters \(R,M\), and allows in particular to bound fractional moments with \(R\ll \sqrt{M}\).
We will thus be able to tailor the scales \((M_i)_{1\leq \ell}\) in order for~\eqref{eq:all-scales} to be satisfied.
We then deduce that the half-moment decays exponentially in the number of scales.

\subsection{Proof of  \texorpdfstring{
    \Cref{thm:half-moment}}{} }
\label{sec:combining}

We now detail how to combine \Cref{prop:key,prop:coarse} in order to derive \Cref{thm:half-moment}.
The key is to choose the scales \(M_1 < \cdots < M_{\ell}\) in \Cref{prop:coarse} so that, applying \Cref{prop:key} on each scale, we obtain~\eqref{eq:all-scales}.
For simplicity, we first treat the case \(\th(N,\beta)\leq 0\); we explain how we get the result in the case \(\th(N,\beta) >0\) afterwards (it combines the proof in the case \(\th(N,\beta) \leq 0 \) together with some estimate from~\cite{BCT25} in the case \(\th(N,\beta)>0\)).

Let us also recall the definition~\eqref{def:theta-N-beta} of \(\th(N,\beta) \coloneqq \pi R_N - \frac{\pi}{\sigma^2(\beta)}\) and note that from the asymptotic \eqref{RN-asym} we have that as \(N,M \to\infty\),
\begin{align}
  \label{eq:th-asym}
  \th(M,\beta)
  = \th(N,\beta) + \pi (R_M-R_N)
  = \th(N,\beta) +\log \Bigl(\frac{M}{N}\Bigr) +o(1),
\end{align}
where the \(o(1)\) is \emph{uniform} in \(\beta\in (0,1)\).

\subsubsection{The case \texorpdfstring{\(\th(N,\beta) \leq 0\)}{}}

First of all, notice that \(\th(M_i-M_{i-1},\beta) \leq 0\) since \(M\mapsto \th(M,\beta)\) is non-decreasing (by the definition~\eqref{def:theta-N-beta}) and \(\th(N,\beta)\leq 0\).
Thus, applying \Cref{prop:key} directly, with \(M=M_i-M_{i-1}\) and \(R=\sqrt{M_{i-1}}\), we get that
\begin{equation}
  \label{eq:apply-key-prop}
  \sup_{f\in\cM_1^{\rm disc}(\sqrt{M_{i-1}})}\EE\Bigl[Z_{M_i-M_{i-1}}^{\beta,\w}(f)^{1/2}\Bigr]
  \leq C_0 \Biggl(1+\log\Biggl(\frac{\log(\frac{M_i-M_{i-1}}{1+M_{i-1}})}{1+\abs{\th(M_i-M_{i-1},\beta)}}\Biggr)\Biggr)^{-1/2}.
\end{equation}

We now make the following choice for the scales in the coarse-graining, in order to make the upper bound in~\eqref{eq:apply-key-prop} uniformly small at all scales (we omit the integer part for simplicity).
For \(i\geq 0\), define for some fixed \(\rho\geq 2\) (chosen large just below)
\begin{equation}
  \label{eq:scale-coarse}
  M_i
  \coloneqq N^{1- \alpha_0 \rho^{-i}}  \qquad \text{ with }  \quad \alpha_0
  \coloneqq \frac{\log \bigl(\frac{N}{1+R^2}\bigr)}{\log N}
\end{equation}
so in particular we indeed have \(M_0 = N^{1-\alpha_0} = 1+R^2\).
Then, we let
\begin{equation}
  \label{def:ell}
  \ell
  = \ell(R,N,\beta)
  \coloneqq \log_{\rho} \Biggl( \frac{\log\bigl(\frac{N}{1+R^2}\bigr)}{1+ \abs{\th(N,\beta)}} \Biggr) \quad \text{which implies that} \quad \alpha_0 \rho^{-\ell}
  = \frac{1+\abs{\th(N,\beta)}}{\log N}.
\end{equation}
One can check that for \(1\leq i\leq \ell\) we have the following scaling relations:
\begin{align}
  \label{eq:ratioM}
  \log \Bigl(\frac{M_i}{M_{i-1}}\Bigr)
   & = \alpha_0 \rho^{-i} (\rho-1) \log N \,,     \\
  \label{eq:ratioMM}
  \log\Bigl(\frac{M_i-M_{i-1}}{1+M_{i-1}}\Bigr)
   & =   \alpha_0 \rho^{-i} (\rho-1) \log N +O(1) \\
  \label{eq:ratioMN}
  \log \Bigl(\frac{M_i-M_{i-1}}{N} \Bigr)
   & = -\alpha_0 \rho^{-i} \log N +O(1)
\end{align}

On the other hand, recalling the relation \eqref{eq:th-asym} and using \eqref{def:ell} and \eqref{eq:ratioMN}, we have (uniformly in \(\beta\in (0,1)\)), since \(\th(M_{i}-M_{i-1},\beta),\th(N,\beta)\leq 0\) and \(M_{i}-M_{i-1}\leq N\),
\begin{align*}
  1+|\th(M_{i}-M_{i-1},\beta)|
   & = 1+|\th(N,\beta)| - \log \Bigl(\frac{M_i-M_{i-1}}{N}\Bigr) + o(1)        \\
   & = \alpha_0\rho^{-\ell}\log N + \alpha_0\rho^{-i} \log N +O(1) \,.
\end{align*}
In particular, we have that \(1+\abs{\th(M_{i}-M_{i-1},\beta)} \leq 2 \alpha_0  \rho^{-i} \log N +C\), so that
\begin{equation}
  \label{eq:th-concl}
  \frac{\log\bigl(\frac{M_i-M_{i-1}}{1+M_{i-1}}\bigr)}{1+ \abs{\th(M_i-M_{i-1},\beta)}}
  \geq  \frac{\alpha_0\rho^{-i}(\rho-1)\log N -C'}{2\alpha_0\rho^{-i}\log N+C'} \geq  c' (\rho-1) \,.
\end{equation}
Going back to~\eqref{eq:apply-key-prop}, we get that uniformly in \(1\leq i \leq \ell\)
\begin{equation}
  \label{eq:frac-moment}
  \sup_{f\in \cM_1^{\rm disc}(\sqrt{M_{i-1}})} \EE\bigl[ Z_{M_i-M_{i-1}}^{\beta,\w}(f)^{1/2} \bigr]
  \leq c \log \bigl( c' (\rho-1)\bigr)^{-1/2} \,.
\end{equation}
Therefore, if we fix \(\rho\) sufficiently large, namely \(\rho = 1+ \e^{(300 c)^2}/c'\), this proves~\eqref{eq:all-scales}.

Noting that \(M_{\ell}\leq N \), we can apply \Cref{prop:coarse} to conclude that
\[
  \sup_{f \in \cM_1^{\disc}(R)} \EE\bigl[Z_{N}^{\beta,\w}(f)^{1/2}\bigr]
  \le 20  \e^{- \ell}
  = 20 \Biggl( \frac{\log\bigl(\frac{N}{1+R^2}\bigr)}{1+ \abs{\th(N,\beta)}} \Biggr)^{-1/\log \rho} \,,
\]
the last identity following from the definition~\eqref{def:ell} of \(\ell = \ell(R,N,\beta)\).
This concludes the proof of \Cref{thm:half-moment} in the case \(\th(N,\beta)\leq0\), recalling the asymptotic~\eqref{eq:asymp-V} of the Volterra function~\(\mathcal{V}(\cdot)\).

\subsubsection{The case \texorpdfstring{\(\th(N,\beta) > 0\)}{}}.
The general procedure is rather similar to the case \(\th(N,\beta) \leq 0\) after reducing,  first,  to \(\th(N,\beta)=0\) and then taking advantage of some monotonicity properties.

To reduce to the case \(\th(N,\beta)=0\), we define \(\beta_c = \beta_c(N)\) by the relation \(\sigma^2(\beta_c) R_N=1\).
Recalling \eqref{def:theta-N-beta}, this implies that \(\th(N,\beta_c) =0\).
From \eqref{def:theta-N-beta} it is also evident that \(\th(N,\beta) > 0\) is equivalent to \(\beta >\beta_c\).
The monotonicity of the fractional moments in \(\beta\) (see \cite[Prop.~3.4]{Z24}) implies that \(\EE\bigl[ Z_{M_i-M_{i-1}}^{\beta,\, \w}(f)^{1/2} \bigr] \leq \EE\bigl[ Z_{M_i-M_{i-1}} ^{\beta_c, \,\w}(f)^{1/2} \bigr]\).

Choose, next,  the scales \((M_i)_{1\leq i\leq \ell}\) as in \eqref{eq:scale-coarse}.
Our choice of \(\beta_c\) such that \(\th(N,\beta_c) =0\) implies that the number of scales is
\begin{align}\label{eq:chooseell}
  \ell
  = \ell(R,N,\beta_c)
  \coloneqq \log_{\rho} \Biggl( \frac{\log\bigl(\frac{N}{1+R^2}\bigr)}{1+ \abs{\th(N,\beta_c)}} \Biggr)
  = \log_{\rho} \Bigl( \log\Bigl(\frac{N}{1+R^2}\Bigr) \Bigr) \,.
\end{align}
Applying \Cref{prop:key} and using \eqref{eq:ratioMM}, we get similarly as~\eqref{eq:frac-moment} that
\begin{align*}
  \sup_{f\in \cM_1^{\rm disc}(\sqrt{M_{i-1}})} \EE\bigl[ Z_{M_i-M_{i-1}}^{\beta,\w}(f)^{1/2} \bigr]
   & \leq c \Biggl( \log\Biggl( \frac{\log\bigl(\frac{M_i-M_{i-1}}{1+M_{i-1}}\bigr)} {1+ \abs{\th(M_i-M_{i-1},\beta_c)}} \Biggr) \Biggr)^{-1/2} \\
   & \leq c \log \bigl(c'(\rho-1)\bigr)^{-1/2} \,.
\end{align*}
We then choose \(\rho\) large enough, so that~\eqref{eq:all-scales} is verified, which then by \eqref{eq:all-scales2} and the choice of \(\ell\) in \eqref{eq:chooseell} leads to
\begin{align}
  \label{eq:posth1}
  \sup_{f \in \cM_1^{\disc}(R)} \EE\bigl[Z_{N}^{\beta,\w}(f)^{1/2}\bigr]
   & \le 	20 \, \e^{- \ell} \sup_{f \in \cM_1^{\disc}(\sqrt{M_{\ell}})} \EE\bigl[Z_{N-M_{\ell}}^{\beta,\w}(f)^{1/2}\bigr] \notag                          \\
   & =	20\,\log\Bigl(\frac{N}{1+R^2}\Bigr)^{-1/\log \rho} \sup_{f \in \cM_1^{\disc}(\sqrt{M_\ell})} \EE\bigl[Z_{N-M_{\ell}}^{\beta,\w}(f)^{1/2}\bigr] .
\end{align}
Furthermore, \(M_{\ell} = N^{1- \frac{1}{\log N}} = \e^{-1} N \leq N/2\).
Thus the monotonicity of \(n\mapsto \EE[Z_{n}^{\beta,\w}(f)^{1/2}]\) (as checked in \Cref{lem:monotone} below) gives
\begin{align}
  \label{eq:posth2}
  \sup_{f \in \cM_1^{\disc}(R)} \EE\bigl[Z_{N}^{\beta,\w}(f)^{1/2}\bigr]
  \leq  20\,  \log\Bigl(\frac{N}{1+R^2}\Bigr)^{-1/\log \rho} \sup_{f \in \cM_1^{\disc}(\sqrt{N/2})} \EE\bigl[Z_{N/2}^{\beta,\w}(f)^{1/2}\bigr] \,.
\end{align}
The conclusion then follows from the estimate from~\cite[Thm.~2.2]{BCT25}, which gives that for \(\th(N,\beta)\geq 0\)
\[
  \sup_{f \in \cM_1^{\disc}(\sqrt{N/2})} \EE\bigl[Z_{N/2}^{\beta,\w}(f)^{1/2}\bigr]
  \leq C \exp\bigl( - c\, \e^{\th(N,\beta)}\bigr)
  \leq C \,  \cV\bigl(\e^{\th(N,\beta)} \bigr) ^{-c} \,
\]
where the second inequality follows from the asymptotics~\eqref{eq:asymp-V} of the Volterra function.
Combined with \eqref{eq:posth1}  and \eqref{eq:posth2} this concludes the proof of \Cref{thm:half-moment} for \(\th(N,\beta) >0\).
\qed

\begin{lemma}
  \label{lem:monotone}
  The sequence  \(n\mapsto Z_{n}^{\beta,\w}(f)\) is non-decreasing for the convex order.
  In particular, \(n\mapsto \EE\bigl[Z_{n}^{\beta,\w}(f)^{1/2}\bigr]\) is non-increasing and \(n\mapsto \EE\bigl[Z_{n}^{\beta,\w}(f)^{2}\bigr]\) is non-decreasing.
\end{lemma}

\begin{proof}
  Let \(g:\R_+ \to \R\) be a convex function.
  By Jensen's inequality, we get that
  \[
    \EE\bigl[ g(Z_{n+1}^{\beta,\omega}(f)) \,\bigm|\, \mathcal{F}_n \bigr]
    \geq g\bigl( \EE[ Z_{n+1}^{\beta,\omega}(f) \,|\, \mathcal{F}_n ] \bigr)
    = g\bigl( Z_{n}^{\beta,\omega}(f)\bigr) \,,
  \]
  where we have used the martingale property of \(Z_{n}^{\beta,\omega}(f)\).
  Taking the expectation again, we obtain that \( \EE\bigl[ g(Z_{n+1}^{\beta,\omega}(f))\bigr] \geq \EE\bigl[g( Z_{n}^{\beta,\omega}(f))\bigr]\), which proves the convex monotonicity.

  The second part of the statement simply follows by applying the previous inequality to \(g(x)=x^2\) or \(g(x)=\sqrt{x}\) (the latter being concave, the monotonicity is reversed).
\end{proof}

\section{Proof of the key \texorpdfstring{\Cref{prop:key}}{the key proposition}}
\label{sec:keyprop}

The proof of \Cref{prop:key} goes through a change of measure argument, which is designed to take into account the interaction of the polymer path with the ambient disorder.
The idea of a change of measure was first introduced in \cite{DGLT09} and then refined in the course of several studies, e.g.
\cite{Lac11, GLT11, BL17, BL18, JL25, BCT25}.
There are two main steps in this methodology: The first one is to `tilt' the disorder law \(\PP\) by a suitable function that encapsulates the interaction between the polymer path and the ambient disorder, and the choice needs to be computationally tractable.
The second step is to perform the estimate.
In our situation we also need an intermediate step, which is to carefully choose the scales.
We carry out these three steps in the following subsections.

\subsection{Change of measure and the proxy}
\label{sec:change-measure}

Recall that if \(f\) is a probability distribution on~\(\Z^2\), then we have that \(\EE[Z_N^{\beta,\w}(f)]=1\).
Since \(Z_N^{\beta,\w}(f)\geq 0\), we can interpret this partition function as a probability density on the \((\Omega, \cF, \PP)\) generated by the disorder variables \(\{\omega(n,x) \colon n\in \N, x\in \Z^2 \}\).
Let us also denote \(\cF_N\coloneqq \sigma( \omega(n,x) \colon n \leq N, x\in \Z^2 )\).

\begin{definition}[Size-biased measure]
  \label{def:sizebiased}
  For every probability mass function \(f\) on \(\Z^2\), we define the \emph{size-biased measure}  on \((\Omega, \cF, \PP)\), corresponding to initial distribution of the polymer~\(f\) by \(\dd \tPP_f = Z_{N}^{\beta,\w}(f) \dd \PP\), \textit{i.e.}\
  \[
    \tPP_f(A)
    = \tPP_{f,N}^{\beta} (A)
    \coloneqq \EE\bigl[ \indic_{A} \, Z_{N}^{\beta,\w}(f) \bigr] \qquad \text{for any } A\in \cF \,.
  \]
\end{definition}
The probability measure \(\tPP_f\) then defines the operators \(\tEE_f[\cdot]\), \(\tVV_f[\cdot]\) and \(\tC_f[\cdot,\cdot]\) as the expectation, variance and covariance under \(\tPP_f\).
The next result shows that  in order to control the half-moment of \(Z_N^{\beta,\w}(f)\), it is enough to construct a proxy \(X(f)\) whose mean has a large shift under the size-biased law, relative to its fluctuations under the two measures.
\begin{lemma}
  \label{lem:change-measure}
  Let \(Z>0\) with \(\EE[Z]=1\) and define the size-biased probability measure \(\dd\tPP=Z\dd\PP\).
  Let \(X\in L^2(\PP)\cap L^2(\tPP)\) and set \(\Delta_X \coloneqq \tEE[X]-\EE[X]\), \(\Sigma_X^2 \coloneqq \Var[X]+\tVV[X]\).
  Then
  \begin{equation}
    \label{eq:size-biased-half-moment}
    \EE\bigl[\sqrt{Z} \bigr]
    \leq \sqrt{\frac{2\Sigma_X^2}{2\Sigma_X^2 + \Delta_X^2 }}\,.
  \end{equation}
\end{lemma}

\begin{proof}
  Set \(\mu \coloneqq\EE[Z^{1/2}]\).
  Then by definition of the size-biased measure, we have
  \[
    \tEE[Z^{-1/2}]
    = \EE[Z^{1/2}]
    = \mu
    \qquad \text{ and } \qquad 
    \tVV[Z^{-1/2}]
    = 1-\mu^2
    = \Var[Z^{1/2}].
  \]
  since \(\tEE[Z^{-1}]= 1=\EE[Z]\).
  Moreover,
  \begin{align*}
    \Cov[X,Z^{1/2}]
     & = \EE[XZ^{1/2}]-\EE[X]\mu,   \\
    \tC[X,Z^{-1/2}]
     & = \tEE[XZ^{-1/2}]-\tEE[X]\mu
    = \EE[XZ^{1/2}]-\tEE[X]\mu.
  \end{align*}
  Consequently, by subtracting the two equalities, we get
  \[
    \Cov[X,Z^{1/2}]-\tC[X,Z^{-1/2}]
    = \mu\bigl(\tEE[X]-\EE[X]\bigr)
    =\mu \Delta_X.
  \]
  Taking the squares of both sides and using \((a-b)^2\leq 2(a^2+b^2)\) gives
  \[
    \mu^2 \Delta_X^2
    \leq 2\bigl(\Cov[X,Z^{1/2}]^2+\tC[X,Z^{-1/2}]^2\bigr).
  \]
  Applying Cauchy--Schwarz under \(\PP\) to the first term and under \(\tPP\) to the second term, we thus get
  \[
    \mu^2\Delta_X^2
    \leq 2\bigl(\Var[X]\Var[Z^{1/2}]+\tVV[X]\tVV[Z^{-1/2}]\bigr)
    = 2\Sigma^2_X (1-\mu^2).
  \]
  Rearranging the terms, we obtain~\eqref{eq:size-biased-half-moment}.
\end{proof}

\begin{remark}
  The inequality \eqref{eq:size-biased-half-moment} is sharp, equality is indeed attained in exactly two cases:
  \begin{itemize}
  \item The constant \(Z\equiv 1\) and any \(X\), since \(\Delta_X\equiv 0\), which makes both sides of \eqref{eq:size-biased-half-moment} equal to \(1\).
  \item The two-point distributions defined by \(\PP(Z=c)=\frac{1}{1+c}\), \(\PP(Z=\frac1c)=\frac{c}{1+c}\) for some \(c\in(0,1)\) and \(X=\mathrm{a} \sqrt{Z} +\mathrm{b}\) for some \(\mathrm{a}\neq 0, \mathrm{b}\in\R\).
  \end{itemize}
  
\end{remark}

We now introduce a proxy \(X(f)\) for the partition function, which we think as a coarse-grained version of the chaos expansion of \(Z_N^{\beta,\w}(f)\) (see \cite[Section~4.3]{BCT25} for a more detailed discussion).
We show below that, indeed, the mean of \(X(f)\) has a large shift under the size-biased law, relative to its fluctuations under the original and size-biased distributions.

We choose scales \(N_0\coloneqq1+R^2<N_1<\cdots <N_k\), for some \(k=k(R,N,\beta)\), which we make explicit below, see~\eqref{def:scales}, and we define
\begin{equation}
  \label{def:Xf}
  X(f)
  \coloneqq \sum_{i=1}^{k} X_i(f) \qquad \text{ with } \quad X_i(f)
  \coloneqq Z_{N_{i-1},N_i}^{\beta,\w}(f)-1 \,,
\end{equation}
where for \(a\leq b\) and \(f:\Z^2\to\R\), we have introduced the notation
\begin{equation}
  \label{eq:partition-function-strip}
  Z_{a,b}^{\beta,\omega}(f)
  \coloneqq \sum_{x\in \Z^2} f(x) \E_x\Bigl[ \exp\Bigl( \sum_{n=a+1}^b \bigl(\beta \omega_{n,S_n} -\lambda(\beta)\bigr)\Bigr)\Bigr] \,.
\end{equation}
In other words, \(Z_{a,b}^{\beta,\w}(f)\) corresponds to the partition function of a polymer which starts at time~\(0\) from an initial distribution \(f\) but moves in an environment where disorder is switched off outside of the interval \((a,b]\).
Note also that \(Z_{a,b}^{\beta,\omega}(f) \stackrel{d}{=} Z_{b-a}^{\beta, \omega} (f\ast q_a)\), where \(q_a(\cdot)\) is the random walk kernel.
This will allow for a smoothing (diffusing) effect that is important in order to keep the fluctuations under control.

Notice that all the \(X_i(f)\) are centered and independent, and they satisfy
\[
  \tEE_f[X_i(f)]
  = \EE[Z_{N}^{\beta,\w}(f)(Z_{N_{i-1},N_i}^{\beta,\w}(f)-1)]
  = \EE[X_i(f)^2]  \,.
\]
In particular, we obtain that
\begin{equation}
  \label{eq:tEE=Var}
  \tEE_f[X(f)]
  = \Var[X(f)]
  = \sum_{i=1}^k \Var\bigl[Z_{N_{i-1},N_i}^{\beta,\w}(f) \bigr] \,.
\end{equation}
We now apply \Cref{lem:change-measure} with \(Z=Z_N^{\beta,\w}(f)\) and \(X=X(f)\).
Since \(\EE[X(f)]=0\), relation \eqref{eq:tEE=Var} gives
\[
  \tEE_f[X(f)]-\EE[X(f)]
  = \Var[X(f)].
\]
Therefore, in view of \Cref{lem:change-measure}, we have that
\begin{equation}
  \label{eq:apply-change-measure}
  \EE\bigl[Z_{N}^{\beta,\w}(f)^{1/2}\bigr]
  \leq \Bigl(\frac{2 \varepsilon_X}{ 2\varepsilon_X +1}\Bigr)^{1/2}  \qquad \text{ with }\ \varepsilon_X
  \coloneqq \frac{1}{\Var[X(f)]} + \frac{\tVV[X(f)]}{\Var[X(f)]^2} \,.
\end{equation}
It thus remains to obtain a lower bound on \(\Var[X(f)]\) and an upper bound on \(\tVV_f[X(f)]\), uniformly over \(f\in\cM_1^{\rm disc}(R)\).

\subsection{Choosing the scales \texorpdfstring{\(N_1<\cdots<N_k\)}{}, variance and size-biased variance}
\label{sec:scales-moment}

We now specify how we choose the scales in the definition~\eqref{def:Xf} of~\(X(f)\) (again, we omit integer parts for simplicity).
We define scales in analogy with~\eqref{eq:scale-coarse} (with here \(\rho=2\)): for all \(0\leq i \leq k\), we set
\begin{equation}
  \label{def:scales}
  N_i
  \coloneqq N^{1-\alpha_0 2^{-i}} \,, \qquad \text{ with } \alpha_0
  \coloneqq \frac{\log(\frac{N}{1+R^2})}{\log N} \,,
\end{equation}
so in particular \(N_0 = 1+R^2\).
From now on, we assume that \(\log\bigl(\frac{N}{1+R^2}\bigr) \geq 2 (1+ \abs{\th(N,\beta)})\), otherwise the statement in \Cref{prop:key} holds trivially.
We then choose the number of scales to be
\begin{equation}
  \label{def:k}
  k
  = k(R,N,\beta)
  \coloneqq \log_{2}\Biggl(\frac{\log(\frac{N}{1+R^2})}{1+\abs{\th(N,\beta)}}\Biggr) \geq 1 \,.
\end{equation}
Let us notice that, as in~\eqref{def:ell}, the choice of \(k\) is made in such a way that \(\alpha_0 2^{-i} \log N \geq \alpha_0 2^{-k} \log N = 1+ \abs{\th(N,\beta)}\) and ensures that \(N_{i}-N_{i-1} \geq \frac12 N_{i}\) for all \(1\leq i \leq k\), see e.g.~\eqref{eq:ratioM} (with \(\rho=2\) here).

\subsubsection*{Variance estimate}

The choice of the scales is in fact made precisely so that the second moment of~\(X_i(f)\) is of constant order, independently of \(1\leq i\leq k\).
\begin{lemma}
  \label{lem:var-Xj}
  There is a constant \(c_0 \in (0,1)\) such that, for any \(f \in \cM_1^{\rm disc}(R)\), if \((N_i)_{1\leq i\leq k}\) are defined as in~\eqref{def:scales} with \(k\) as in \eqref{def:k}, for any \(1\leq i \leq k\) we have
  \[
    c_0
    \leq \Var[Z_{N_{i-1},N_i}^{\beta,\w}(f)] = \Var[X_i(f)]
    \leq \frac{1}{c_0} \,.
  \]
\end{lemma}

\begin{proof}
  The proof relies on a general second moment estimate that we prove in \Cref{sec:second-moment} (together with~\Cref{prop:var-polymer}).
  \begin{lemma}
    \label{lem:var-strip}
    There is a constant \(c \in (0,1)\) such that, for any \(f\in \cM_1^{\rm disc}(R)\), for any \(b> a \geq 1+ R^2\), with \(b\geq 2a\) we have
    \[
      c \log\Bigl(\frac{b}{a}\Bigr) \, \cV\Bigl(  c\, \e^{\th(b,\beta)} \Bigr)
      \leq  \Var\bigl[Z_{a,b}^{\beta,\omega}(f)  \bigr]
      \leq \frac{1}{c} \log\Bigl(\frac{b}{a}\Bigr) \, \cV\Bigl(  c^{-1} \e^{\th(b,\beta)} \Bigr)\, \,.
    \]
  \end{lemma}
  Using \Cref{lem:var-strip} with \(a=N_{i-1}\), \(b=N_i\), and since \(\frac{c}{1+|\th(N_i,\beta)|} \leq \cV(c \e^{\th(N_i,\beta)}) \leq \frac{c'}{1+|\th(N_i,\beta)|}\)  by~\eqref{eq:asymp-V} (recall that we are in a regime where \(\th(N_i,\beta)\leq 0\)), we get that
  \[
    c \frac{\log\bigl(\frac{N_i}{N_{i-1}}\bigr)}{1+|\th(N_i,\beta)|}
    \leq \Var[Z_{N_{i-1},N_i}^{\beta,\w}(f)]
    \leq c' \frac{\log\bigl(\frac{N_i}{N_{i-1}}\bigr)}{1+|\th(N_i,\beta)|} \,.
  \]
  By the choice of the scales we have that \(\log (\frac{N_i}{N_{i-1}}) = 2^{-i} \alpha_0 \log N\) and also \(\th(N_i,\beta) = \th(N,\beta) + \log(\frac{N_i}{N}) +o(1)\), recall~\eqref{eq:th-asym}.
  Given our choice \(\th(N,\beta)\leq 0\) and \(\log(\frac{N_i}{N}) = - 2^{-i}\alpha_0 \log N\), we get that
  \[
    \frac{\log\bigl(\frac{N_i}{N_{i-1}}\bigr)}{1+|\th(N_i,\beta)|}
    = 	\frac{2^{-i} \alpha_0 \log N}{|\th(N,\beta)| + 2^{-i}\alpha_0 \log N +o(1)}\,.
  \]
  Since \(\alpha_0 2^{-i} \log N \geq 1+|\th(N,\beta)|\) for any \(1\leq i\leq k\) (see below~\eqref{def:k}), we get that this is bounded from above and from below by a constant, which concludes the proof.
\end{proof}

Recalling~\eqref{eq:tEE=Var}, \Cref{lem:var-Xj} directly gives that there is a constant \(c_0\) such that
\begin{equation}
  \label{inf-variance}
  \inf_{f\in \cM_1^{\rm disc}(R)} \Var[X(f)]
  \geq c_0 k \,.
\end{equation}

\subsubsection*{Size-biased variance estimates}

The main idea is to split \(\tVV_f[X(f)]\) as the sum of a \textit{diagonal} contribution and an \textit{off-diagonal} one: we write
\begin{equation}
  \label{eq:split-size-biased-var}
    \tVV_f[X(f)]
    =  \sum_{i=1}^{k} \tVV_f[X_{i}(f)] + 2 \sum_{1\leq i < j \leq k} \tC_f\bigl[X_{i},X_{j}\bigr] \,.
\end{equation}
The remainder of the proof consists of proving the following estimates.

\begin{lemma}[Diagonal terms]
  \label{lem:diag}
  There is a constant \(C_0'>0\) such that, for any \(1\leq i \leq k\)
  \[
    \sup_{f\in \cM_1^{\rm disc}(R)}\tVV_f[X_i(f)]
    \leq C_0' \,.
  \]
\end{lemma}

\begin{lemma}[Off-diagonal terms]
  \label{lem:off-diag}
  There is a constant \(C_0''>0\) such that, for any \(1\leq i< j \leq k \), we have
  \[
    \sup_{f\in \cM_1^{\rm disc}(R)} \tC_f\bigl[X_{i}(f),X_{j}(f)\bigr]
    \leq C_0''  \frac{N_{i}}{N_{j-1}}\,.
  \]
\end{lemma}

We can now estimate the size-biased variance.
The sum of the diagonal terms in~\eqref{eq:split-size-biased-var} is clearly bounded by \(C_0' k\) thanks to \Cref{lem:diag}.
For the off-diagonal terms in \eqref{eq:split-size-biased-var}, recalling the definition~\eqref{def:scales} of \(N_i\), we notice that
\[
  \frac{N_{i}}{N_{j-1}}
  = N^{-\alpha_0 (2^{-i} - 2^{-(j-1)})}
  \leq
  \begin{cases}
    1                     & \quad \text{ if } j
    =i+1  \,                                    \\
    N^{-\frac12 \alpha_0 2^{-i}}
    \leq \e^{- 2^{k-i-1}} & \quad \text{ if } j
    \geq i+2\,,
  \end{cases}
\]
where we have used that \(\alpha_0 2^{-k} > \frac{1+\abs{\th(N,\beta)}}{\log N} \geq \frac{1}{\log N}\) so that \(N^{\alpha_0 2^{-k}} \geq \e\).
Thus, the sum of the off-diagonal terms is bounded by \(C_0''\) times
\[
  \sum_{i =1}^{k-1}  1 + \sum_{i =1}^{k-2}  \sum_{j=i+2}^k \e^{- 2^{k-i}}
  \leq k \Bigl(1+ \sum_{i =1}^{k-2}  \e^{- 2^{k-i-1}} \Bigr)
  \leq 2 k \,.
\]
All together, we obtain that there is a constant \(C_0>0\) such that 
\begin{equation}
  \label{sup-size-biased-variance}
  \sup_{f\in \cM_1^{\rm disc}(R)} \tVV_f[X(f)] \leq C_0 k \,.
\end{equation}

\subsubsection*{Conclusion of the proof}

Plugging the bounds~\eqref{inf-variance} and~\eqref{sup-size-biased-variance} into \eqref{eq:apply-change-measure}, we obtain that
\[
  \sup_{f\in\cM_1^{\rm disc}(R)} \EE\bigl[Z_N^{\beta,\w}(f)^{1/2}\bigr]
  \leq \frac{C}{\sqrt{1+k}}
\]
with \(C\coloneqq\max\bigl\{1,\frac{\sqrt{2(c_0+C_0)}}{c_0}\bigr\}\).
This concludes the proof of \Cref{prop:key}.
\qed

\subsection{Proof of \texorpdfstring{\Cref{lem:diag,lem:off-diag}}{}}

We now prove the two main estimates on the diagonal part and the off-diagonal part of the sized biased variance; the second part is the most delicate one.

\begin{proof}[Proof of \Cref{lem:diag}]
  The proof essentially follows from \Cref{lem:var-Xj} and the hypercontractivity property of \cite[Thm.~1.11]{CSZ25}.
  First of all, writing \(Z_i(f)\coloneqq X_i(f)+1 = Z_{N_{i-1},N_i}^{\beta,\w}(f)\) for simplicity, notice that a simple calculation gives (recall \(\EE[Z_i(f)]=1\))
  \[
    \tVV_f[X_i(f)]
    = \EE[ Z_i(f) X_i(f)^2 ] - \EE[Z_i(f)X_i(f)]^2
    = \EE[Z_i(f)^3] - \EE[Z_i(f)^2]^2
    \leq \EE[Z_i(f)^3]\,.
  \]
  Recall also that \(Z_i(f) \stackrel{d}{=} Z_{N_i-N_{i-1}}^{\beta,\w}(f_{i-1})\) with \(f_{i-1} \coloneqq f\ast q_{N_{i-1}}\).
  Then \cite[Thm.~1.11]{CSZ25} says that whenever the distribution \(f\) is such that the second moment of \(Z_N(f)\) is bounded, then its higher moments are bounded by suitable powers of its second moment.
  In combination with  \Cref{lem:var-Xj}, this implies that there exists a uniform constant \(\mathfrak{C}_3\) such that
  \[
    \EE[Z_i(f)^3]
    \leq \mathfrak{C}_3 \,\EE[Z_i(f)^2]^{3/2}
    \leq \mathfrak{C}_3 \, (c_0)^{-3/2} \,.
  \]
  This holds uniformly for \(f\in \cM_1^{\rm disc}(R)\), which concludes the proof of \Cref{lem:diag}.
\end{proof}

Before we turn to the proof of \Cref{lem:off-diag}, let us introduce the so-called chaos expansion of the partition function.
One may write
\[
  \exp\Bigl(\sum_{n=1}^{N} (\beta \omega(n,S_n) -\lambda(\beta))\Bigr)
  = \prod_{n=1}^{N} \prod_{z\in \Z^2} \bigl(1+\xi(n,z) \indic_{\{S_n=z\}} \bigr) \,,
\]
with \(\xi(n,z) \coloneqq e^{\beta \omega(n,z) -\lambda(\beta)}-1\).
Then, performing a binomial expansion of the product, we obtain that the directed polymer partition function~\eqref{eq:polymeas} admits the following chaos expansion
\[
  Z_{N}^{\beta,\w}(f)
  =	\sum_{k\geq 0} \sumtwo{1\leq n_1<\cdots<n_k \leq N}{x_0,x_1,\ldots,x_k \,\in \,\Z^2} f(x_0) \prod_{i=1}^{k} q_{n_i-n_{i-1}}(x_i-x_{i-1}) \prod_{i=1}^k \xi(n_i,x_i),
\]
where \(q_n(z) \coloneqq \P(S_n=z)\) is the two-dimensional simple random walk transition probability; by convention \(n_0=0\).
To shorten these types of expressions, we introduce and use the following notation: for a multi-index set \(A\coloneqq \{(n_1,x_1),\ldots,(n_{j},x_{j})\} \in \bigl( \N \times \Z^2\bigr)^{j}\) with \(n_1<\cdots <n_j\), we denote
\begin{align*}
  q^{(f)}(A)
  \coloneqq  \sum_{x_0\in \Z^2} f(x_0) \prod_{i =1}^{j} q_{n_i-n_{i-1}}(x_i-x_{i-1}) \qquad \text{and} \qquad \xi(A)
  \coloneqq \prod_{i =1}^{j} \xi(n_i,x_i) \,.
\end{align*}
We also set by convention \(q^{(f)}(\emptyset)=1\), \(\xi(\emptyset) =1\) and \(q^{(f)}(A)=0\) if \(A\subset \N\times \Z^2\) contains two elements with the same time-index.
Then, for \(f:\Z^2\to \R_+\) the above chaos expansion can be rewritten as
\begin{align}
  \label{def:multi-chaos}
  Z_{N}^{\beta,\w}(f)
  =	\sum_{A\subset \llb 1,N\rrb \times \Z^2} q^{(f)}(A) \, \xi(A) \,.
\end{align}
Similarly, recalling the notation~\eqref{eq:partition-function-strip}, for \(a<b\) we can write
\[
  Z_{a,b}^{\beta,\w}(f)
  = \sum_{A\subset \llb a+1,b\rrb \times \Z^2} q^{(f)}(A) \, \xi(A)
\]
and also
\begin{equation}
  \label{eq:chaos-Xi}
  X_i(f)
  \coloneqq \sum_{A \in \cI_i} q^{(f)}(A) \xi(A) \,,\quad \text{ with } \quad  \cI_i
  \coloneqq \bigl\{ A\subset \llb N_{i-1}+1,N_{i}\rrb \times \Z^2, A\neq\emptyset \bigr\} \,.
\end{equation}

Finally, let us stress that since the \(\xi(n,x)\) are centered and independent, the variables \(\xi(A)\) are orthogonal in \(L^2(\PP)\).
Note also that a simple calculation gives that \(\EE[\xi(n,x)^2] = \e^{\lambda(2\beta)-2\lambda(\beta)}-1 \eqqcolon \sigma^2(\beta)\) so we have that
\begin{equation}
  \label{eq:orthogonality}
  \EE[\xi(A)\xi(B)]
  = \sigma^2(\beta)^{\abs{A}}\, \indic_{\{A=B\}} \,.
\end{equation}
In particular, using the chaos decomposition, we get that 
\begin{equation*}
  \EE\bigl[Z_N^{\beta,\omega}(f)^2\bigr]
    = \sum_{A\subset \llb 1,N\rrb \times \Z^2} \sigma^2(\beta)^{|A|}q^{(f)}(A)^2 \,.
\end{equation*}
Notice also that, of \(A = \{(n_1,x_1), \ldots, (n_j,x_j)\}\) with \(n_1<\cdots <n_j\), we get that if \(j\geq 1\),
\[
q^{(f)}(A) = q^{(f)}_{n_1}(x_1) q(\hat{A}) \qquad \text{ with } \hat{A} \coloneqq \{(n_2-n_1,x_2-x_1), \ldots, (n_j-n_1,x_j-x_1)\} \,,
\]
where we have used the notation \(q^{(f)}_{n_1}(x_1) = \sum_{x_0\in \Z^2} f(x_0)q_{n_1}(x_1-x_0)\).
Thus, decomposing over the first point in \(A\) (for non-empty~\(A\)), we get, for any \(f\) a probability distribution:

\[
\EE\bigl[Z_N^{\beta,\omega}(f)^2\bigr] = 1+ \sum_{n_1=1}^{N} \sum_{x_1 \in \Z^2} \sigma^2(\beta)^{|A|}q^{(f)}_{n_1}(x_1) q^{(f)}_{n_1}(x_1) \sum_{\hat{A} \subset \llb 1, N-n_1 \rrb \times \Z^2} q(\hat{A})^2 \,.
\]
This leads to the following ``first collision decomposition'' 
\begin{equation}
  \label{eq:first-collision}
  \begin{split}
    \EE\bigl[Z_N^{\beta,\omega}(f)^2\bigr]
    =& 1+\sum_{j=1}^N q_j(f,f) \sigma^2(\beta) \EE\bigl[(Z_{N-j}^{\beta,\omega})^2\bigr]\,,\\
    \text{ with }\ q_j(f,f) & \coloneqq \sum_{x,y \in \Z^2} f(x)f(y)q_{2j}(x-y) \,,
  \end{split}
\end{equation}
where we have used that \(\EE\bigl[(Z_{m}^{\beta,\omega})^2\bigr] = \sum_{A\subset \llb 1,m\rrb \times \Z^2} \sigma^2(\beta)^{|\hat{A}|} q(\hat{A})^2\) and the Chapman--Kolmogorov property to see that \(\sum_{x_1\in \Z^2} q_j(x_1-x) q_j(x_1-y) = q_{2j}(x-y)\).

\begin{proof}[Proof of \Cref{lem:off-diag}]
  The idea is to use the chaos expansion and orthogonality to write the covariance in terms of collision diagrams.
  Using the chaos expansion~\eqref{eq:chaos-Xi}, we obtain
  \[
    \tC_f \bigl[X_{i}(f),X_{j}(f)\bigr]
    = \sum_{A_1\in \cI_{i}, A_2\in \cI_{j}} q^{(f)}(A_1) q^{(f)}(A_2) \tC_f \bigl[\xi(A_1),\xi(A_2)\bigr] \,.
  \]
  Now, given the chaos expansion~\eqref{def:multi-chaos} of \(Z_{N}^{\beta,\w}(f)\) and by the orthogonality relation~\eqref{eq:orthogonality}, we have that
  \[
    \tEE_f[\xi(A)]
    = \EE[ Z_{N}^{\beta,\w}(f)\xi(A)]
    = \sigma^2(\beta)^{\abs{A}} q^{(f)}(A) .
  \]
  Since the time intervals of \(\cI_{i}\) and \(\cI_{j}\) are disjoint we have \(\xi(A_1)\xi(A_2) = \xi(A_1\cup A_2)\) and we hence obtain that
  \[
    \tC_f \bigl[\xi(A_1),\xi(A_2)\bigr]
    = \sigma^2(\beta)^{\abs{A_1}+\abs{A_2}}  \bigl(q^{(f)}(A_1\cup A_2) - q^{(f)}(A_1) q^{(f)}(A_2) \bigr) \,.
  \]

  Let us denote by \((a_i,x_i)\) and \((b_i,y_i)\) the first and last point of \(A_i\) for \(i=1,2\).
  Then, notice that we can write
  \[
    q^{(f)}(A_1\cup A_2)
    = q^{(f)}(A_1) \,q_{a_2-b_1}(x_2-y_1) \,q(\hat A_2) \quad \text{ and }\quad q^{(f)}(A_2)
    = q_{a_2}^{(f)}(x_2) \, q( \hat A_2) \,,
  \]
  where if \(A_2 = \{ (m_1,z_1), \ldots, (m_j,z_j) \}\) with \(m_1<\cdots <m_j\) (and with \((a_2,x_2)= (m_1,z_1)\)) we have denoted 
  \[
  q(\hat{A}_2) = \prod_{i=2}^j q_{m_i-m_{i-1}}(z_i-z_{i-1}) \,.
  \]
  In other words \(q(\hat{A}_2)\) is the kernel associated with the set \(A_2\) translated by its first point with this point being removed, namely the set \(\hat A_2 = \{(m_2-m_1,z_2-z_1), \ldots, (m_{j}-m_1,z_j-z_1) \}\).
  We refer to Figure~\ref{fig:kernels} for an illustration.
  All together, we have 
  \begin{equation}
    \label{eq:difference-covariance}
    \begin{split}
      q^{(f)}(A_1)q^{(f)}(A_2) \, & \Bigl( q^{(f)}(A_1\cup A_2)-q^{(f)}(A_1)q^{(f)}(A_2)\Bigr) \\
    &= q^{(f)}(A_1)^2 \, q_{a_2}^{(f)}(x_2) \,\Bigl( q_{a_2-b_1}(x_2-y_1)-q_{a_2}^{(f)}(x_2) \Bigr) \, q(\hat A_2)^2 \,.
    \end{split}
  \end{equation}

  \begin{figure}
  \centering
  \begin{tikzpicture}[scale=1, line cap=round, line join=round]
      \tikzstyle{blackpath}=[thick, blue!40!black]
      \tikzstyle{greenpath}=[thick, green!40!black]
      \tikzstyle{dot}=[fill=black, circle, inner sep=2.2pt]
      \coordinate (Start) at (0, 2);     
      \coordinate (P1) at (3.5, 3.6);
      \coordinate (P2) at (4, 3.8);
      \coordinate (P3) at (4.5, 3.6);
      \coordinate (P4) at (5, 3.7);
      \coordinate (P5) at (5.5, 3.4);
      \coordinate (P6) at (8.5, 1.9);
      \coordinate (P7) at (9, 2.3);
      \coordinate (P8) at (9.5, 2.4);
      \coordinate (P9) at (10, 1.9);
      \coordinate (P10) at (10.5, 2.3);
      \coordinate (P11) at (11, 2.1);
      \draw[blackpath] (Start) to[out=10, in=190] (P1);
      \draw[blackpath] (P1) -- (P2) -- (P3) -- (P4) -- (P5);
      \draw[blackpath] (P5) to[out=-10, in=170] (P6);
      \draw[blackpath] (P6) -- (P7) -- (P8) -- (P9) -- (P10) -- (P11);
      \coordinate (G1_Start) at (0, 1.6);
      \coordinate (G1_P1) at (3.5, 3.52);
      \coordinate (G1_P2) at (4, 3.72);
      \coordinate (G1_P3) at (4.5, 3.52);
      \coordinate (G1_P4) at (5, 3.62);
      \coordinate (G1_P5) at (5.5, 3.32);
      \draw[greenpath] (G1_Start) to[out=0, in=210] (G1_P1);
      \draw[greenpath] (G1_P1) -- (G1_P2) -- (G1_P3) -- (G1_P4) -- (G1_P5);
      \coordinate (G2_Start) at (0, 1.4);
      \coordinate (G2_P6) at (8.5, 1.82);
      \coordinate (G2_P7) at (9, 2.22);
      \coordinate (G2_P8) at (9.5, 2.32);
      \coordinate (G2_P9) at (10, 1.82);
      \coordinate (G2_P10) at (10.5, 2.22);
      \coordinate (G2_P11) at (11, 2.02);
      \draw[greenpath] (G2_Start) to[out=-10, in=170] (G2_P6);
      \draw[greenpath] (G2_P6) -- (G2_P7) -- (G2_P8) -- (G2_P9) -- (G2_P10) -- (G2_P11);
      \foreach \p in {P1, P2, P3, P4, P5, P6, P7, P8, P9, P10, P11} {
          \node[dot] at (\p) {};
      }
      \node[above left=0.1cm, font=\small] at (P2) {$A_1$};
      \node[above right=-0.05cm, font=\small] at (P5) {$(b_1, y_1)$};
      \node[below=0.01cm, font=\small] at (P6) {$(a_2, x_2)$};
      \node[below right=0.05cm, font=\small] at (P9) {$A_2$};
  \end{tikzpicture}
  \caption{Illustration of the formula~\eqref{eq:difference-covariance}. 
  The kernel \(q^{(f)}(A_1\cup A_2)\) is represented as a blue line connecting the dots (the elements of \(A_1\cup A_2\)), each line corresponding to a random walk transition kernel \(q_{n-m}(y-x)\).
  The kernels \(q^{(f)}(A_1)\) and \(q^{(f)}(A_2)\) are represented by green lines.
  The figure illustrates how the formula~\eqref{eq:difference-covariance} factorizes.}
  \label{fig:kernels}
  \end{figure}
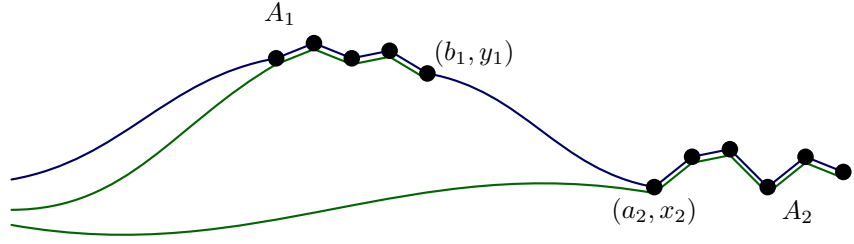

  Summing over \(\hat A_2\) (with \((a_2,x_2)\) fixed), we have that
  \[
    \sum_{\hat A_2 \subset \llb a_2+1,N_{j}\rrb \times \Z^2} \sigma^2(\beta)^{\abs{\hat A_2}} q(\hat A_2)^2
    = 	\EE\bigl[(Z_{N_{j}-a_2}^{\beta,\w})^2\bigr] \,,
  \]
  so that \(\tC_f \bigl[X_{i}(f),X_{j}(f)\bigr]\) is equal to
  \[
    \sum_{A_1 \in \cI_{i}} \sigma^2(\beta)^{\abs{A_1}}  q^{(f)}(A_1)^2 \!\! \sum_{a_2 =N_{j-1}+1}^{N_j} \!\! \sigma^2(\beta) \EE\bigl[(Z_{N_{j}-a_2}^{\beta,\w})^2\bigr]  \biggl(\sum_{ x_2 \in \Z^2 } q_{a_2}^{(f)}(x_2) \bigl(q_{a_2-b_1}(x_2-y_1)-q_{a_2}^{(f)}(x_2)\bigr) \biggr) \,.
  \]
  where we recall that \((b_1,y_1)\) is the last point of \(A_1\).
  Now, using the Chapman--Kolmogorov property, we get that
  \[
    \sum_{ x_2 \in \Z^2 } q_{a_2}^{(f)}(x_2) \bigl(q_{a_2-b_1}(x_2-y_1)-q_{a_2}^{(f)}(x_2)\bigr)
    = q^{(f)}_{2a_2-b_1}(y_1) - q_{2a_2}(f,f) \,,
  \]
  where \(q_{m}(f,f) \coloneqq \sum_{x,x'\in \Z^2} f(x) q_m(x'-x) f(x')\).
  Then, we can use directly \cite[Eq.~(6.4)]{BCT25}, which shows that, uniformly in \(y_1\), we have
  \[
    q^{(f)}_{2a_2-b_1}(y_1) - q_{2a_2}(f,f)
    \leq C \frac{b_1}{(a_2)^2}
    \leq C' \frac{N_{i}}{N_{j-1}} q_{2a_2}(f,f)
  \]
  where for the second inequality we have used that \(b_1\leq N_{i}\), that \(a_2 \geq N_{j-1}\) and that \(q_{2a_2}(f,f) \geq \frac{c}{a_2}\) by the local limit theorem (using that \(f\in \cM_1^{\rm disc}(R)\) and \(a_2 \geq N_0 =1+R^2\)).

  All together, we have that there is a constant \(C>0\) such that
  \[
    \begin{split}
      \tC_f \bigl[X_{i}(f),X_{j}(f)\bigr]
       & \leq C' \frac{N_{i}}{N_{j-1}} \sum_{A_1 \in \cI_{i}} \sigma^2(\beta)^{\abs{A_1}}  q^{(f)}(A_1)^2 \sum_{a_2 =N_{j-1}+1}^{N_j} q_{2a_2}(f,f)\sigma^2(\beta) \EE\bigl[(Z_{N_{j}-a_2}^{\beta,\w})^2\bigr] \\
       & = C' \frac{N_{i}}{N_{j-1}} \, \Var[Z_{N_{i-1},N_i}^{\beta,\w}(f)] \, \Var[Z_{N_{j-1},N_j}^{\beta,\w}(f)] \,.
    \end{split}
  \]
  For the last identity, we have noticed that the sum over \(a_2\) is exactly the first collision decomposition of \(\Var[Z_{N_{j-1},N_j}^{\beta,\w}(f)]\), see~\eqref{eq:first-collision}, and that the sum over \(A_1\) is precisely \(\Var[Z_{N_{i-1},N_i}^{\beta,\w}(f)]\), recall~\eqref{eq:chaos-Xi}.
  Then, using \Cref{lem:var-Xj}, we conclude the proof of \Cref{lem:off-diag}.
\end{proof}

\section{The coarse-graining estimate}
\label{sec:coarse}

In this section, we prove \Cref{prop:coarse}.
The proof is rather standard e.g. \cite[App.~A]{BCT25}, but the coarse-graining is here inhomogeneous.
Recall that we fix some scales \(M_0\coloneqq 1+R^2 <M_1<\cdots <M_{\ell} \leq N\); for simplicity of notation, we assume that \(\sqrt{M_i}\) are integers.

For every \(0\leq i\leq \ell\), we partition \(\Z^2\) into \(L^{\infty}\) balls of radius \(\sqrt{M_i}\).
More precisely, for \(y\in \Z^2\), let us denote by \(B_i(y) \coloneqq B(2y\sqrt{M_i}, \sqrt{M_i})\) the \(L^{\infty}\) ball centered at \(2y\sqrt{M_i}\).

Then, for any ``skeleton'' \(\cY_{\ell} = (y_i)_{0\leq i \leq \ell} \in (\Z^2)^{\ell+1}\), we define the coarse-grained partition function starting from \(f\in \cM_1^{\disc}(R)\) and with skeleton~\(\cY_{\ell}\), \textit{i.e.}\ visiting the \(B_i(y_i)\) at time \(M_i\):
\[
  Z_{M_{\ell}}^{\beta,\w}(f ; \cY_{\ell})
  = \sum_{x_0 \in B_0(y_0)} \ \sum_{x_1 \in B_1(y_1)} \cdots \sum_{x_\ell \in B_{\ell}(y_\ell)}  Z_{0,M_0}^{\beta,\omega}(f,x) \prod_{i=1}^{\ell} Z_{M_{i-1},M_i}^{\beta,\omega}(x_{i-1},x_i) \,,
\]
where \(Z_{a,b}^{\beta,\w} (x,x')\) is the point-to-point partition function from \((a,x)\) to \((b,x')\), namely
\[
  Z_{a,b}^{\beta,\w} (x,x')
  \coloneqq \E\Bigl[ \exp\Bigl(\sum_{n=a+1}^{b} (\beta \w(n,S_n) -\lambda(\beta)) \Bigr)  \indic_{\{S_b =x'\}}\;\Bigm|\; S_a =x\Bigr] \,,
\]
and \(Z_{a,b}^{\beta,\omega}(f,x) = \sum_{z\in \Z^2} f(z) Z_{a,b}^{\beta,\omega}(z,x)\).
(We stress that the notation is slightly different from~\eqref{eq:partition-function-strip}.)
For any integers \(a<b\) and any probability measure \(\mu\) on \(\Z^2\) and any \(B\subset \Z^2\), we also introduce the notation
\[
  Z_{s,t}^{\beta,\w} (\mu ; B)
  \coloneqq \sum_{x\in \Z^2}\sum_{x'\in B} \mu(x)  Z_{s,t}^{\beta,\w} (x,x') \,,
\]
which is the partition function of a polymer with initial distribution \(\mu\) at time \(a\) and constrained to end in \(B\) at time \(b\).
With this notation, we can write that
\[
  Z_{N}^{\beta,\w}(f)
  =  \sum_{\cY_{\ell} \in (\Z^2)^{\ell} }Z_{M_{\ell}}^{\beta,\w}(f ; \cY_{\ell}) Z_{M_{\ell},N}^{\beta,\w}\bigl(\mu_{\cY_{\ell}}^{\beta,\w} ; \Z^2 \bigr) \,,
\]
where \(\mu_{\cY_{\ell}}^{\beta,\w}\) is the endpoint distribution of the polymer measure associated with the partition function \(Z_{M_{\ell}}^{\beta,\w}(f; \cY_{\ell})\) (hence it has support in \(B_{\ell}(y_{\ell})\)).
Using the standard inequality \((\sum_i z_i)^{1/2} \leq \sum_i z_i^{1/2}\) for non-negative \((z_i)\), we then get that
\[
  Z_{N}^{\beta,\w}(f)^{1/2}
  \leq  \sum_{\cY_{\ell} \in (\Z^2)^{\ell} }Z_{M_{\ell}}^{\beta,\w}(f ; \cY_{\ell} )^{1/2} Z_{M_{\ell},N}^{\beta,\w}\bigl(\mu_{\cY_{\ell}}^{\beta,\w} ; \Z^2 \bigr)^{1/2} \,.
\]
Taking the conditional expectation with respect to \(\mathcal{F}_{M_{\ell}}\), using translation invariance and since \(\mu_{\cY_{\ell}}^{\beta,\w}\) is \(\cF_{M_{\ell}}\)-measurable, we get that
\[
  \EE\Bigl[ Z_{M_{\ell},N}^{\beta,\w}\bigl(\mu_{\cY_{\ell}}^{\beta,\w} ; \Z^2 \bigr)^{1/2} \; \Bigm| \; \cF_{M_{\ell}} \Bigr]
  \leq  \sup_{\mu \in \cM_1^{\disc}(\sqrt{M_{\ell}})} \EE\bigl[Z_{N-M_{\ell}}^{\beta,\w}(\mu)^{1/2}\bigr] \,,
\]
so that
\begin{equation}
  \label{eq:uppercoarse}
  \EE\bigl[ Z_{N}^{\beta,\w}(f)^{1/2} \bigr]
  \leq  \sum_{\cY_{\ell} \in (\Z^2)^{\ell} } \EE\bigl[ Z_{M_{\ell}}^{\beta,\w}(f ; \cY_{\ell} )^{1/2} \bigr]  \sup_{\mu \in \cM_1^{\disc}(\sqrt{M_{\ell}})} \EE\bigl[Z_{N-M_{\ell}}^{\beta,\w}(\mu)^{1/2}\bigr] \,.
\end{equation}
We are therefore reduced to estimating a fractional moment along a skeleton \(\cY_{\ell}\).

With the notation introduced above, we have
\[
  Z_{M_{\ell}}^{\beta,\w}(f;\cY_{\ell})
  = Z_{M_{\ell-1}}^{\beta,\w}(f; \cY_{\ell-1}) \; Z_{M_{\ell-1},M_{\ell}}^{\beta,\w}\bigl(\mu_{\cY_{\ell-1}}^{\beta,\w} ; B_{\ell}(y_{\ell}) \bigr)\,,
\]
so that taking the conditional expectation with respect to \(\cF_{M_{\ell-1}}\), we get
\[
  \EE\Bigl[ Z_{M_{\ell}}^{\beta,\w}(f;\cY_{\ell})^{1/2} \; \Bigm| \; \cF_{M_{\ell-1}} \Bigr]
  \leq  Z_{M_{\ell-1}}^{\beta,\w}(f; \cY_{\ell-1})^{1/2}\!\!\!  \sup_{\mu \colon  \mathrm{supp}(\mu) \subset B_{\ell-1}(y_{\ell-1})}\!\!\! \EE\Bigl[ Z_{M_{\ell-1},M_{\ell}}^{\beta,\w}\bigl(\mu ; B_{\ell}(y_{\ell}) \bigr)^{1/2} \Bigr] \,,
\]
where the supremum is over probability distributions~\(\mu\).
Therefore, if we define for all \(1\leq i\leq \ell\)
\[
  \mathcal{Q}_{i}(x,y)
  \coloneqq \sup_{\mu \colon  \mathrm{supp}(\mu) \subset B_{i-1}(x)} \EE\Bigl[  Z_{0,M_i-M_{i-1}}^{\beta,\w}\bigl(\mu ;  B_{i}(y) \bigr)^{1/2} \Bigr] \,,
\]
and \(\mathcal{Q}_{0}(y) \coloneqq \sup_{f \in \cM_1^{\disc}(R)} \EE\bigl[  Z_{0,M_0}^{\beta,\w}\bigl(f ;  B_{0}(y) \bigr)^{1/2} \bigr]\),
then we get by iteration that
\[
  \sup_{f \in \cM_1^{\disc}(R)} \EE\Bigl[Z_{M_{\ell}}^{\beta, \w}(f;\cY_{\ell})^{1/2} \Bigr]
  \leq \mathcal{Q}_{0}(y_0) \prod_{i=1}^\ell \mathcal{Q}_i(y_{i-1},y_{i})  \,.
\]
Therefore, plugging into~\eqref{eq:uppercoarse} and using translation invariance we get that
\[
\begin{split}
    \sup_{f \in \cM_1^{\disc}(R)} &\EE\bigl[Z_{N}^{\beta, w}(f)^{1/2}\bigr] \\
    & \leq \sum_{(y_0,y_1,\ldots, y_\ell) \in (\Z^2)^{\ell+1}} \mathcal{Q}_{0}(y_0)  \prod_{i=1}^\ell\mathcal{Q}_i(y_{i-1},y_{i}) \times \sup_{f \in \cM_1^{\disc}(\sqrt{M_{\ell}})} \EE\bigl[Z_{N-M_{\ell}}^{\beta,\w}(f)^{1/2}\bigr]  \,.
\end{split}
\]

Since
\[
  \sum_{(y_0,y_1,\ldots, y_\ell) \in (\Z^2)^{\ell+1}} \mathcal{Q}_{0}(y_0) \prod_{i=1}^\ell \mathcal{Q}_i(y_{i-1},y_{i})
  \leq \biggl( \sum_{y \in \Z^2} \mathcal{Q}_0(y)\biggr) \prod_{i=1}^{\ell}\biggl( \sup_{x\in \Z^2} \sum_{y \in \Z^2} \mathcal{Q}_i(x,y) \biggr) \,,
\]
it only remains to show that \(\sum_{y \in \Z^2} \mathcal{Q}_0(y)\leq 20\) and that thanks to the condition~\eqref{eq:all-scales} we have
\begin{equation}
  \label{eq:coarse-one-scale}
  \forall 1 \leq i\leq \ell \,,
  \qquad \sup_{x\in \Z^2} \sum_{y \in \Z^2} \mathcal{Q}_i(x,y)
  \leq \e^{-1} \,.
\end{equation}

In order to prove~\eqref{eq:coarse-one-scale}, note that, by translation invariance, we may reduce to considering the case \(2x\sqrt{M_{i-1}}\leq \sqrt{M_i}\).
Simply applying Jensen's inequality, we have that for \(\mu\) supported in \(B_{i-1}(x)\),
\[
  \begin{split}
    \EE\bigl[  Z_{0,M_{i}-M_{i-1}}^{\beta,\w}\bigl(\mu ;  B_i(y) \bigr)^{1/2} \bigr]
     & \leq  \Bigl(\sum_{ z \in B_{i-1}(x)} \mu(z) \P_z\bigl(S_{M_{i}-M_{i-1}} \in B_i(y)\bigr) \Bigr)^{1/2}    \\
     & \leq \P\Bigl(S_{M_{i}-M_{i-1}} \in 2y \sqrt{M_i} + \llb -2\sqrt{M_i},2\sqrt{M_i}\llb^2 \Bigr)^{1/2} \, ,
  \end{split}
\]
where we have widened the ball around \(2y\sqrt{M_i}\) by \(\sqrt{M_i}\) to account for the worst case scenario for the starting point \(z\).
Now, we have that for all \(y\) such that \(|y|_1>2\), it holds:
\[
  \P\Bigl(S_{M_i} \in 2y \sqrt{M_i} + \llb -2\sqrt{M_i}, 2\sqrt{M_i}\llb^2 \Bigr)
  \leq \P \Bigl( \mathrm{SRW}_{M_i}
  \geq (2\abs{y}_1-4)\sqrt{M_i} \Bigr)
  \leq 2 \e^{- 2 (\abs{y}_1-2)^2} \,,
\]
where \(\mathrm{SRW}_n\) denotes the position at time \(n\) of a one-dimensional, simple, symmetric random walk, the first inequality follows by projecting the 2D random walk \(S_M\) on the diagonals \(e_1+e_2\) and \(e_1-e_2\), where \(e_1,e_2\) are the basis vectors of \(\R^2\) and the second inequality is the standard inequality \(\P \bigl( |\mathrm{SRW}_n | \geq k\bigr)\leq 2 \exp\bigl( -\frac{k^2}{2n}\bigr) \).

On the other hand, by assumption~\eqref{eq:all-scales} and translation invariance, we have that for any \(y\in \Z^2\), for any \(1\leq i \leq \ell\),
\[
  \mathcal{Q}_i(x,y)
  \leq \sup_{\mu \in \cM_1^{\disc}(\sqrt{M_{i-1}})} \EE\Bigl[  Z_{M_i-M_{i-1}}^{\beta,\w}(\mu)^{1/2} \Bigr]
  \leq \frac{1}{300} \,.
\]
Therefore, for any integer threshold \(K\ge 1\), we obtain that
\[
  \sup_{x\in \Z^2} \sum_{y \in \Z^2} \mathcal{Q}_i(x,y)
  \leq \sum_{\abs{y}_1 \leq K} \frac{1}{300} + \sum_{\abs{y}_1 > K}  \bigl(2 \e^{- 2(\abs{y}_1-2)^2}\bigr)^{1/2} \,.
\]
Now, it turns out that for \(K=4\) the sum of the two terms is smaller than \(\e^{-1}\).
This concludes the proof of~\eqref{eq:coarse-one-scale}.

On the other hand, using Jensen's inequality to get that \(\mathcal{Q}_0(y)\leq 1\) for \(\abs{y}_1 \leq 2\) and \(\mathcal{Q}_0(y)\leq (2\e^{-2 (\abs{y}_1-2)^2})^{1/2}\) for \(\abs{y}_1>2\) similarly as above (recall that \(M_0=1+R^2\) and that the initial distribution is \(f\in \mathcal{M}_1^{\mathrm{disc}}(R)\)), we get that 
\[
  \sum_{y \in \Z^2} \mathcal{Q}_0(y) \leq \sum_{\abs{y}_1 \leq 2} 1 + \sum_{\abs{y}_1 > 2}  \bigl(2 \e^{- 2(\abs{y}_1-2)^2}\bigr)^{1/2} \leq 20 \,,
\]
where the last inequality follows from a numerical evaluation of the sums.
This concludes the proof of \Cref{prop:coarse}.
\qed

\section{Second moment estimates}
\label{sec:second-moment}

This section collects several estimates on second moments of the partition function.
In particular, we prove \Cref{prop:var-SHF,prop:var-polymer} as well as \Cref{lem:var-strip}.
As a key step of the proof of \Cref{prop:var-polymer}, we prove a sharp estimate on the second moment of the point-to-plane partition function, see \Cref{lem:var-point-to-plane} below.
Many of the estimates are classical and can be found in the literature, but one key feature of \Cref{prop:var-polymer} (and \Cref{lem:var-point-to-plane}) that requires additional work is that the estimates are uniform in the parameters and in particular hold in the subcritical, critical and supercritical regime.

\subsection{Proof of \texorpdfstring{\Cref{prop:var-SHF}}{}}

Notice that, with \(G_{\th}(u) \coloneqq \int_0^{\infty} \frac{\e^{(\th -\gamma)s} s u^{s-1}}{\Gamma(s+1)} \dd s\) we have that \(\int_{0}^t G_{\th}(u) \dd u = \cV( t\, \e^{\th-\gamma})\).
Then, in view of \cite[Thm.~6.1]{CSZ25-icm}, we have that
\[
  \Cov(\mathscr{Z}_{t}^{\th}(\phi_1), \mathscr{Z}_{t}^{\th}(\phi_2))
  = \pi \int_0^t g_{2s}(\phi_1,\phi_2)  \cV( (t-s) \e^{\th-\gamma} )\dd s \,,
\]
with
\[
  g_{2s}(\phi_1,\phi_2)
  \coloneqq \int_{\R^2\times\R^2} \phi_1(x_1) \phi_2(x_2) g_{2s}(x_2-x_1) \dd x_1 \dd x_2 \,.
\]
\Cref{prop:var-SHF} then follows easily, by using that \(\cV( (t-s) \e^{\th-\gamma} ) \leq \cV(t\e^{\th-\gamma} )\) for the upper bound, and restricting the integral to \(0\le s\le(1-\e^{\gamma-1})t\) where \(\cV( (t-s) \e^{\th-\gamma} ) \geq \cV(t\e^{\th-1} )\) for the lower bound.
(Note that in \Cref{prop:var-SHF} we can thus choose  \(c_2=\e^{-1}\) and \(c_4=\e^{-\gamma}\).)
\qed

\subsection{Proof of \texorpdfstring{\Cref{prop:var-polymer,lem:var-strip}}{}}

The proofs will rely on a sharp estimate of the point-to-plane partition function that we state in the following lemma, of independent interest.
Its proof is postponed to the next subsection and requires to deal with some technical subtleties in order to obtain estimates that are uniform in the parameters.

\begin{lemma}
  \label{lem:var-point-to-plane}
  Define \(\cV_{N}^{\beta} \coloneqq \sigma^2(\beta) \EE\bigl[ (Z_N^{\beta,\w})^2 \bigr]\).
  Then, recalling the definition~\eqref{def:theta-N-beta} of \(\th(N,\beta)\), we have
  \begin{equation}
    \label{eq:sharp-second-moment}
    \cV_N^{\beta}
    = \pi\cV\Bigl( (1+o(1))\e^{\th(N,\beta)-\gamma} \Bigr)  \quad \text{ as } N\to\infty, \beta\downarrow 0\,.
  \end{equation}
  In particular \(\cV_N^{\beta}\) converges to \(\pi\cV(\e^{\th-\gamma})\) if \(\lim_{N\uparrow\infty,\beta\downarrow 0}\th(N,\beta) = \th \in \R\).
  More generally, there are constants \(c_1,c_2\) \(c_3,c_4\) such that, for any \(N\in \N\) and \(\beta \in (0,1)\),
  \begin{equation}
    \label{eq:second-moment-cst}
    c_1\,\cV\bigl( c_2\, \e^{\th(N,\beta)} \bigr)
    \leq \cV_{N}^{\beta}
    \leq c_3\,\cV\bigl( c_4\, \e^{\th(N,\beta)} \bigr) \,.
  \end{equation}
\end{lemma}

We are now ready to prove \Cref{prop:var-polymer,lem:var-strip}.

\begin{proof}[Proof of \Cref{prop:var-polymer}]
  The main idea to compute the variance is use the first collision decomposition~\eqref{eq:first-collision}. 
  Recalling the definition \(\cV_M^{\beta} \coloneqq \sigma^2(\beta) \EE\bigl[(Z_{M}^{\beta,\omega})^2\bigr]\) and the fact that \(M\mapsto \cV_M^{\beta}\) is non-decreasing (see \Cref{lem:monotone}), we easily have that
  \begin{equation}
    \label{eq:bounds-variance}
    \sum_{j=1}^{N/2} q_j(f,f)  \times \cV_{N/2}^{\beta}
    \leq \Var[Z_N^{\beta,\omega}(f)]
    \leq \sum_{j=1}^{N} q_j(f,f) \times \cV_N^{\beta} \,,
  \end{equation}
  where we recall that \(q_j(f,f)\) is defined in~\eqref{eq:first-collision}.
  Thanks to \Cref{lem:var-point-to-plane}, noting also that \(\th(N/2,\beta) \geq \th(N,\beta)- cst.\) (see \eqref{eq:th-asym}), we have that (possibly renaming the constants for the lower bound)
  \[
    c_1\, \cV\bigl(c_2 \,\e^{\th(N,\beta)}\bigr)
    \leq \cV_{N/2}^{\beta}
    \leq \cV_N^{\beta}
    \leq c_3\, \cV\bigl(c_4 \,\e^{\th(N,\beta)}\bigr) \,.
  \]
  It thus only remains to estimate \(\sum_{j=1}^{M} q_j(f,f)\) with \(M=N/2\), \(M=N\).
  Note that we have the following standard estimates on \(q_j(f,f)\) when \(f=U_R\).

  \begin{claim}
    There are two constants \(c,c'\) such that
    \[
      \frac{c'}{\max(j,1+R^2)}
      \leq q_j(U_R,U_R)
      \leq \frac{c}{\max(j,1+R^2)} \,.
    \]
  \end{claim}

  \begin{proof}
    For the upper bound, recalling the definition~\eqref{eq:first-collision} of \(q_{j}(f,f)\), we write
    \[
      q_j(U_R,U_R)
      \leq \sum_{x,y\in B(R)} \frac{c}{R^2} \frac{c}{R^2} q_{2j}(x-y)
      \leq \frac{c''}{R^2} \sum_{z\in B(2R)} q_{2j}(z)
      = \frac{c''}{R^2} \P(|S_{2j}|
      \leq 2R)
    \]
    where in the change of variable \(z=x-y\), we have used that \(|z|\leq 2R\) and that \(\sum_{x\in B(R)} \frac{c}{R^2} \leq c'\).
    We now conclude by noticing that \(\P(|S_{2j}| \leq 2R) \leq c (\frac{R^2}{j} \wedge 1)\), either by the local CLT if \(j>R^2\) or by the CLT if \(j\leq R^2\).

    The lower bound follows similarly: restricting the sum to \(|x|\leq R/2\) and \(|x-y| \leq R/2\), we get that
    \[
      q_j(U_R,U_R)
      \geq \sum_{|x|\leq R/2, |y-x| \leq R/2} \frac{c}{R^2} \frac{c}{R^2} q_{2j}(y-x)
      \geq \frac{c'}{R^2} \P(|S_{2j}|
      \leq R/2) \,.
    \]
    Again, we conclude by noticing that \(\P(|S_{2j}| \leq R/2)\geq c (\frac{R^2}{j} \wedge 1 )\), by the local CLT if \(j>R^2\) or by the CLT if \(j\leq R^2\).
  \end{proof}

  \noindent
  With this claim at hand, we easily get that:
  \begin{itemize}
    \item In the case \(N\geq R^2\), \(\sum_{j=1}^{N} q_j(U_R,U_R) \leq  c\bigl( 1 + \sum_{j=R^2+1}^N \frac{1}{j} \bigr) \leq c \bigl(1+ \log ( \frac{N}{1+R^2}) \bigr)\);
    \item In the case \(N<R^2\), \(\sum_{j=1}^{N} q_k(U_R,U_R) \leq  c \, \frac{N}{1+R^2}\).
  \end{itemize}
  This can be summarized in the bound, valid in all regimes:
  \[
    \sum_{j=1}^{N} q_j(U_R,U_R)
    \leq C \log \Bigl( 1+ \frac{N}{1+R^2}\Bigr) \,.
  \]
  A similar lower bound follows analogously.
  We conclude that
  \[
    c \log \Bigl( 1+ \frac{N}{1+R^2}\Bigr)  \cV_{N/2}^{\beta}
    \leq \Var[Z_N^{\beta}(U_R)]
    \leq  C \log \Bigl( 1+ \frac{N}{1+R^2}\Bigr)  \cV_N^{\beta}\,,
  \]
  which concludes the proof of \Cref{prop:var-polymer}.
\end{proof}

\begin{remark}
  In fact, we can deal with more general initial distributions \(f\) in \Cref{prop:var-polymer}: in view of~\eqref{eq:bounds-variance} and thanks to \Cref{lem:var-point-to-plane}, one only needs to get upper and lower bounds on the quantity \(\sum_{j=1}^M q_j(f,f)\).
\end{remark}

\begin{proof}[Proof of \Cref{lem:var-strip}]
  Recall the notation~\eqref{eq:partition-function-strip} for the partition function \(Z_{a,b}^{\beta,\omega}(f)\) with disorder only in the strip \((a,b]\).
  Let \(f\in \cM_1^{\rm disc}(R)\) and \(b> a \geq 1+R^2\) and notice that \(Z_{a,b}^{\beta,\omega}(f)\) has the same distribution as \(Z_{b-a}^{\beta,\omega}(f\ast q_{a})\).

  Therefore, denoting \(L \coloneqq b-a\) and noticing that \(q_{j}(f\ast q_a,f\ast q_a) = q_{j+a}(f,f)\) by the Chapman--Kolmogorov property, we can apply~\eqref{eq:bounds-variance} to get that
  \[
    \sum_{j=a+1}^{a+L/2} q_j(f,f) \cV_{L/2}^{\beta}
    \leq \Var\bigl[Z_{a,b}^{\beta,\omega}(f)  \bigr]
    \leq \sum_{j=a+1}^{a+L} q_j(f,f) \cV_L^{\beta} \,.
  \]
  Again, thanks to \Cref{lem:var-point-to-plane}, we have that (possibly renaming the constants for the lower bound)
  \[
    c_1 \cV\bigl(c_2 \,\e^{\th(L,\beta)}\bigr)
    \leq \cV_{L/2}^{\beta}
    \leq \cV_L^{\beta}
    \leq c_3 \cV\bigl(c_4 \,\e^{\th(L,\beta)}\bigr) \,.
  \]
  Now, by the local CLT, we have that \(\frac{c}{j} \leq q_{2j}(x-y) \leq \frac{c'}{j}\) uniformly for \(x,y\in B(0,R)\cap \Z^2\) with the same parity and \(j\geq R^2\), so that we get \(\frac{c}{j} \leq q_j(f,f) \leq \frac{c'}{j}\) uniformly for \(f\in \cM_1^{\rm disc}(R)\) and \(j\geq R^2\).
  All together, recalling that \(a+L/2 = \frac{a+b}{2}\) and \(a+L = b\), we get that
  \[
    c' \log \Bigl(\frac{a+b}{2 a}\Bigr)
    \leq \sum_{j=a+1}^{\frac{a+b}{2}} \frac{c}{j}
    \leq \sum_{j=a+1}^{a+L/2} q_j(f,f)
    \leq \sum_{j=a+1}^{a+L} q_j(f,f)
    \leq \sum_{j=a+1}^{b} \frac{c'}{j}
    \leq c'' \log \Bigl(\frac{b}{a}\Bigr) \,,
  \]
  which shows that uniformly for \(f\in \cM_1^{\rm disc}(R)\) and \(b> a \geq 1+R^2\)
  \[
    c' \log \Bigl(\frac{a+b}{2 a}\Bigr) \cV\bigl(c_1 \,\e^{\th(b-a,\beta)}\bigr)
    \leq \Var\bigl[Z_{a,b}^{\beta,\omega}(f)  \bigr]
    \leq c'' \log \Bigl(\frac{b}{a}\Bigr) \cV\bigl(c_2 \,\e^{\th(b-a,\beta)}\bigr) \,.
  \]

  In the case where \(b\geq 2 a\), we simply use that \(\frac{a+b}{2 a} \geq \frac{b}{2a}\) to bound \(\log(\frac{a+b}{2 a}) \geq c\log(\frac{b}{a})\).
  Also, recalling \eqref{eq:th-asym}, we have \(\th(b-a,\beta) = \th(b,\beta) + \log (\frac{b-a}{b}) +O(1)\), with \(0\geq \log (\frac{b-a}{b}) \geq -\log 2\).
  This concludes the proof of \Cref{lem:var-strip}.
\end{proof}

\subsection{Proof of \texorpdfstring{\Cref{lem:var-point-to-plane}}{}}

Note that we do not need to treat the critical case \(\th(N,\beta)\to \th\in \R\) since in that case \cite[Prop.~A.1]{CSZ19-Dickman} gives that
\[
  \mathcal{V}_N^{\beta}
  \coloneqq \sigma^2(\beta) \EE\bigl[ (Z_N^{\beta,\omega})^2\bigr] \xrightarrow[\;N\to\infty\;]{} \pi \int_0^{\infty} \frac{\e^{(\th-\gamma)s}}{\Gamma(s+1)} \dd s
  \eqqcolon \pi \mathcal{V}(\e^{\th-\gamma}) \,,
\]
noting also that \(\sigma^2(\beta)\sim R_N^{-1} \sim \frac{\pi}{\log N}\) in this regime.

Additionally, in the case \(\beta \in (\beta_0,1)\) for some fixed \(\beta_0>0\), the proof of~\eqref{eq:second-moment-cst} is easy. 
Indeed, consider the Lyapunov exponent \(\mathtt{F}_2(\beta) = \lim_{N\to\infty} \frac1N \log\EE[ (Z_N^{\beta,\omega})^2]\).
Then, by monotonicity of \(\beta\mapsto \EE[(Z_N^{\beta,\omega})^2]\) and \(N\mapsto \EE[(Z_N^{\beta,\omega})^2]\), we have \(\e^{c_0 N \mathtt{F}_2(\beta)}\leq \e^{\frac12 N \mathtt{F}_2(\beta_0)}\leq \EE[ (Z_N^{\beta,\omega})^2] \leq \e^{N \mathtt{F}_2(\beta)}\) uniformly for \(\beta\in (\beta_0,1)\) and \(N\) large enough (how large depends on \(\beta_0\)).
Note that we may also bound \(c \e^{\th(N,\beta)} \leq N \mathtt{F}_2(\beta) \leq C \e^{\th(N,\beta)}\) for any \(\beta\in (0,1)\), recalling the definition~\eqref{def:theta-N-beta} of \(\th(N,\beta)\) and the fact that \(\mathtt{F}_2(\beta) \sim 16 \,\e^{-\pi} \e^{-\pi/\sigma^2(\beta)}\) as \(\beta\downarrow 0\), by \cite[Thm.~1.3]{BN26}.
Since \(\mathcal{V}(T) \sim \e^{T}\) as \(T\to\infty\), see~\eqref{eq:asymp-V}, this proves~\eqref{eq:second-moment-cst} for \(\beta\in (\beta_0,1)\), with constants depending on \(\beta_0\).

\smallskip

We now therefore focus on showing~\eqref{eq:sharp-second-moment} when \(N\to\infty\) and \(\beta\leq \beta_0\) for some sufficiently small \(\beta_0>0\); in practice, we consider \(\beta \downarrow 0\) in the following.
All together we thus need to treat the subcritical and supercritical regimes, \textit{i.e.}\ respectively \(\th(N,\beta)\to 0\) and \(\th(N,\beta) \to +\infty\), in a regime \(N\to\infty\), \(\beta\downarrow 0\).
Our first task is to re-write \(\mathcal{V}_N^{\beta}\), in a slightly different way than in \cite{CSZ19-Dickman}.
Using the chaos expansion~\eqref{def:multi-chaos}, \cite{CSZ19-Dickman} proves that
\begin{equation}
  \label{eq:decomp-second-moment}
  \mathcal{V}_N^{\beta}
  =  \sigma^2(\beta) \sum_{k=0}^{\infty} \sum_{\ell_1+\cdots+\ell_k\le N} \sigma^2(\beta)^{k} \prod_{i=1}^k q_{2\ell_i}(0)\,,
\end{equation}
recalling that \(q_{n}(x)\) is the random walk transition kernel.
For any \(\lambda>0\), let us introduce
\[
  \Rhat(\lambda)
  \coloneqq\sum_{n=1}^{\infty}\e^{-\lambda n} q_{2n}(0)\,,
\]
and let \(T_1,T_2,\ldots\) be i.i.d.\ random variables taking values in \(\{1,2,\ldots\}\) with law
\[
  \mathrm P_\lambda(T
  =n)
  = \frac{1}{\Rhat(\lambda)} \e^{-\lambda n} q_{2n}(0) \,.
\]
We also set \(\tau_k=T_1+\cdots+T_k\).
Now, note that for \(n_1+\cdots+n_k\le N\) we have
\[
  \prod_{i=1}^k q_{2\ell_i}(0)
  = \Rhat(\lambda)^k \e^{\lambda(n_1+\cdots+n_k)} \mathrm P_\lambda(T_1
  =\ell_1,\ldots,T_k
  =\ell_k) \,.
\]
Thus, after summing in~\eqref{eq:decomp-second-moment} we obtain the following soft-tilt identity: for any \(\beta>0\) and any \(\lambda>0\),
\begin{equation}
  \label{eq:soft-tilt-identity-0}
  \mathcal{V}_N^{\beta}
  = \sigma^2(\beta)\,\e^{\lambda N} \sum_{k=0}^{\infty} \bigl(\sigma^2(\beta)\Rhat(\lambda)\bigr)^k \mathrm E_\lambda \Bigl[ \e^{-\lambda(N-\tau_k)} \indic_{\{\tau_k\le N\}} \Bigr] \,.
\end{equation}
Note that in \cite{CSZ19-Dickman} (and later, e.g.\ in the variance estimate of \cite{BCT25}), a hard cut-off was used in the definition of the renewal process \(\tau\) instead of the soft tilt, but the latter is more convenient for our purposes.

Since \(\Rhat\) is continuous and strictly decreasing from \(+\infty\) to \(0\) it is invertible, for any \(\beta>0\) we can define \(\mathtt{F}(\beta)\) as the solution of \(\sigma^2(\beta) \hat{R}\bigl(\mathtt{F}(\beta)\bigr) =1\) (it is closely related to the Lyapunov exponent~\(\mathtt{F}_2(\beta)\) introduced above, see \cite[Eq.~(2.6)]{BN26}). 
We then introduce the following parameter:
\begin{equation}
  \label{def:rho-N-beta}
  \rho
  = \rho(N,\beta)
  \coloneqq N \mathtt{F}(\beta)  \quad \text{ or equivalently } \quad \sigma^2(\beta) \Rhat(\rho/N)
  =1 \,.
\end{equation}
We stress that \(\rho(N,\beta)\) will play a similar role as the parameter \(\th(N,\beta)\) in the sense that it is an interpolating parameter, but this new parameter will be more convenient for our purpose.

Since \(\pi R_N = \log N +\alpha+o(1)\) as \(N\to\infty\), see~\eqref{RN-asym}, we have by \cite[Thm.~3.9.1]{BGT89} that
\begin{equation}
  \label{eq:Rhat-asymptotic}
  \pi \Rhat(\lambda)
  = - \log \lambda +\alpha-\gamma+o(1) \qquad \text{ as } \lambda\downarrow 0\,.
\end{equation}
In particular, by definition of \(\mathtt{F}(\beta)\) we get that \(\frac{\pi}{\sigma^2(\beta)} = -\log(\mathtt{F}(\beta)) + \alpha-\gamma +o(1)\) as \(\beta\downarrow 0\).
Recalling the definition~\eqref{def:theta-N-beta} of \(\th(N,\beta)\) and the asymptotics~\eqref{RN-asym}, we thus get that, as \(N\to \infty\) and \(\beta\downarrow 0\)
\[
  \th(N,\beta)
  = \log N +\alpha + \log(\mathtt{F}(\beta)) - \alpha+\gamma +o(1)
  = \log (\rho(N,\beta)) +\gamma +o(1) \,.
\]
All together we get that in all subcritical, critical and supercritical regimes,
\begin{equation}
  \label{eq:rho-theta}
  \rho(N,\beta)
  = \e^{\th(N,\beta) - \gamma +o(1)} \qquad \text{ as } \quad N\to\infty,\, \beta \downarrow 0\,.
\end{equation}
With this notation, \eqref{eq:soft-tilt-identity-0} can be rewritten as:
\begin{equation}
  \label{eq:soft-tilt-identity}
  \mathcal{V}_N^{\beta}
  = \frac{\e^{\lambda N}}{\Rhat(\rho/N)} \sum_{k=0}^{\infty} \Bigl( \frac{\Rhat(\lambda)}{\Rhat(\rho/N)}\Bigr)^k \mathrm E_\lambda \Bigl[ \e^{-\lambda(N-\tau_k)} \indic_{\{\tau_k\le N\}} \Bigr]  \,,
\end{equation}
which supports our claim that \(\rho(N,\beta)\) is natural interpolating parameter.

\subsubsection*{Upper bounds in \Cref{lem:var-point-to-plane}}

Since \(\e^{-\lambda(N-\tau_k)}\le 1\) on the event \(\{\tau_k\le N\}\), we obtain from~\eqref{eq:soft-tilt-identity} the general upper bound
\begin{equation}
  \label{eq:general-soft-upper}
  \mathcal{V}_N^{\beta}
  \le \frac{\e^{\lambda N}}{\Rhat(\rho/N)} \sum_{k=0}^{\infty} \Bigl( \frac{\Rhat(\lambda)}{\Rhat(\rho/N)}\Bigr)^k
  = \frac{\e^{\lambda N}}{\Rhat(\rho/N)-\Rhat(\lambda)} \,,
\end{equation}
where the last identity holds for \(\lambda> \rho/N\), since then we have \(\Rhat(\lambda)< \Rhat(\rho/N)\).
We now treat the supercritical regime \(\rho(N,\beta) \to \infty\) and the subcritical regime \(\rho(N,\beta)\downarrow 0\) separately.

\smallskip
\textbullet\ 
In the supercritical regime \(\rho =\rho(N,\beta)\to\infty\), we fix \(\varepsilon>0\) and we apply \eqref{eq:general-soft-upper} with \(\lambda=\frac{(1+\varepsilon)\rho}{N}\).
Noting that \(\rho/N \to 0\) (since \(\beta\downarrow0\), recall~\eqref{def:rho-N-beta}), by~\eqref{eq:Rhat-asymptotic} we get 
\[
  \Rhat(\rho/N)-\Rhat((1+\varepsilon)\rho/N)
  = \frac{\log(1+\varepsilon)+o(1)}{\pi} \,.
\]
Thus,~\eqref{eq:general-soft-upper} yields that \(\mathcal{V}_N^{\beta}\le \frac{\pi+o(1)}{\log(1+\varepsilon)} \e^{(1+\varepsilon)\rho}\).
Since \(\varepsilon>0\) is arbitrary, we conclude that
\[
  \mathcal{V}_N^{\beta}
  \le \pi \, \e^{(1+o(1))\rho}
  = \pi\, \mathcal{V}\bigl((1+o(1)) \rho \bigr) \,,
\]
using that \(\mathcal{V}(T) \sim \e^{T}\) as \(T\to\infty\), see~\eqref{eq:asymp-V}.
Recalling that \(\rho=(1+o(1))\e^{\th(N,\beta)-\gamma}\), see~\eqref{eq:rho-theta}, this gives the desired upper bound in the supercritical regime \(\th(N,\beta) \to\infty\).

\smallskip
\textbullet\ In the subcritical regime \(\rho=\rho(N,\beta)\to0\), let us set \(\varepsilon_\rho\coloneqq\frac{1}{\log(1/\rho)}\): it verifies \(\varepsilon_\rho\downarrow0\) and \(\varepsilon_\rho>\rho\) for~\(\rho\) small.
We the apply~\eqref{eq:general-soft-upper} with \(\lambda =\varepsilon_\rho /N > \rho/N \).
Using again~\eqref{eq:Rhat-asymptotic}, we get
\[
  \Rhat(\rho/N)-\Rhat(\varepsilon_\rho/N)
  = \frac{\log(\varepsilon_\rho/\rho)+o(1)}{\pi}
  = (1+o(1)) \frac{\log(1/\rho)}{\pi},
\]
so that~\eqref{eq:general-soft-upper} yields that \(\mathcal{V}_N^{\beta}\le  (1+o(1)) \frac{\pi}{\log(1/\rho)}\, \e^{\varepsilon_\rho}\).
All together, since \(\varepsilon_{\rho} \downarrow0\), we get that
\[
  \mathcal{V}_N^{\beta}
  \le (1+o(1)) \frac{\pi}{\log(1/\rho)}
  =  \pi\, \mathcal{V}\bigl((1+o(1)) \rho \bigr) ,
\]
using that \(\mathcal{V}(T) \sim \frac{1}{\log(1/T)}\) as \(T\to0\), see~\eqref{eq:asymp-V}.
Recalling that \(\rho=(1+o(1))\e^{\th(N,\beta)-\gamma}\), see~\eqref{eq:rho-theta}, this gives the desired upper bound in the subcritical regime \(\th(N,\beta) \to -\infty\).

\subsubsection*{Lower bounds in \Cref{lem:var-point-to-plane}}

The idea is again to apply~\eqref{eq:soft-tilt-identity} with some well-chosen \(\lambda\) and to restrict the sum to a subset of indices.
We again treat the supercritical regime \(\rho(N,\beta) \to \infty\) and the subcritical regime \(\rho(N,\beta)\downarrow 0\) separately.

\smallskip
\textbullet\ In the supercritical regime \(\rho=\rho(N,\beta) \to \infty\), we use \eqref{eq:soft-tilt-identity} with the critical tilt \(\lambda = \rho/N\) and we obtain
\begin{equation}
  \label{eq:before-restriction}
  \mathcal{V}_N^{\beta}
  = \frac{\e^{\rho}}{\Rhat(\rho/N)}\, \sum_{k=0}^{\infty} \mathrm E_{\rho/N} \Bigl[ \e^{-\rho(1-\tau_k/N)} \indic_{\{\tau_k\le N\}} \Bigr]
  \geq \frac{\e^{\rho -3\delta \rho}}{\Rhat(\rho/N)}\, \sum_{k=0}^{\infty}\mathrm P_{\rho/N} \bigl((1-3\delta)N\le \tau_k\le N\bigr) \,,
\end{equation}
where we have set \(\delta = \rho^{-1/3}\).
We now restrict the sum to indices \(k\) in the set
\[
  K_\rho
  \coloneqq \biggl\{ k\in\N: (1-2\delta)\frac{N}{\mathrm E_{\rho/N}[T_1]}\le  k\le (1-\delta)\frac{N}{\mathrm E_{\rho/N}[T_1]} \biggr\}.
\]
Note that we have the following moment estimates
\begin{equation}
  \label{eq:moments-T}
  \mathrm E_{\rho/N}[T_1]
  = -\frac{\Rhat'(\rho/N)}{\Rhat(\rho/N)}
  = \frac{(1+o(1))\, N}{\rho\log(N/\rho)}\,, \qquad \mathrm E_{\rho/N}[T_1^2]
  = \frac{\Rhat''(\rho/N)}{\Rhat(\rho/N)}
  = \frac{(1+o(1))\, N^2}{\rho^2\log(N/\rho)} \,,
\end{equation}
where we have used that \(\pi\Rhat(\rho/N) = (1+o(1)) \log(N/\rho)\) by~\eqref{eq:Rhat-asymptotic}, and the following asymptotics for the derivatives of \(\Rhat\):
\[
  \Rhat'(\lambda)
  = -\frac{1+o(1)}{\pi\lambda}, \qquad \Rhat''(\lambda)
  = \frac{1+o(1)}{\pi\lambda^2} \qquad \text{ as } \lambda\downarrow 0 \,.
\]
In particular, it gives that 
\[
|K_{\rho}| = \frac{\delta N}{\mathrm E_{\rho/N}[T_1]} = (1+o(1)) \delta \rho \log(N/\rho) \,.
\]
Additionally for \(k\in K_\rho\) it holds that \((1-2\delta)N\le k\mathrm E_{\rho/N}[T_1]\le (1-\delta)N\) and \(\mathrm{Var}_{\rho/N}(\tau_k)\le C\frac{N^2}{\rho}\).
Therefore Chebyshev's inequality gives that uniformly for \(k\in K_\rho\),
\[
  \mathrm P_{\rho/N} \bigl((1-3\delta)N\le \tau_k\le N\bigr)
  \ge 1-\frac{\mathrm{Var}_{\rho/N}(\tau_k)}{\delta^2N^2}
  \ge 1-\frac{C}{ \delta^2 \rho}
  =1-C\rho^{-1/3}.
\]

All together, going back to~\eqref{eq:before-restriction}, we get that 
\[
  \mathcal{V}_N^{\beta}
  \geq (1-C\rho^{-1/3}) \e^{\rho -3\delta \rho}  \frac{|K_{\rho}|}{\hat{R}(\rho/N)} = (1+o(1)) \pi \, \delta \rho\, \e^{(1-o(1)) \rho} \,,
\]
where we have again used that \(\pi\Rhat(\rho/N) = (1+o(1)) \log(N/\rho)\) by~\eqref{eq:Rhat-asymptotic}.
Since \(\delta\rho = \e^{o(\rho)}\), this therefore gives
\[
  \mathcal{V}_N^{\beta}
  \ge \pi\, \e^{(1+o(1))\rho}
  = \pi\, \mathcal{V}\bigl((1+o(1)) \rho \bigr) \,,
\]
using again that \(\mathcal{V}(T) \sim \e^{T}\) as \(T\to\infty\), see~\eqref{eq:asymp-V}.
Recalling that \(\rho=(1+o(1))\e^{\th(N,\beta)-\gamma}\), see~\eqref{eq:rho-theta}, this gives the desired lower bound in the supercritical regime \(\th(N,\beta) \to\infty\).

\smallskip

\textbullet\ In the subcritical regime \(\rho = \rho(N,\beta) \to 0\), we use \eqref{eq:soft-tilt-identity} with \(\lambda=1/N>\rho/N\).
Observing that \(\e^{-(1-\tau_k/N)}\ge \e^{-1}\) on the event \(\{\tau_k\le N\}\), we thus get
\begin{equation}
  \label{eq:sum-restriction-1}
  \mathcal{V}_N^{\beta}
  \geq \frac{1}{\Rhat(\rho/N)} \sum_{k=0}^{L_N} \Bigl( \frac{\Rhat(1/N)}{\Rhat(\rho/N)}\Bigr)^k \mathrm P_{1/N}(\tau_k\le N)  \,,
\end{equation}
where we have also restricted the sum to \(k\leq L_N\), for some sequence \((L_N)_{N\geq 1}\).
The sequence \((L_N)_{N\geq 1}\) is chosen such that 
\[
\frac{L_N}{\log N} \to 0 
\qquad \text{ and }\qquad
L_N \frac{\log(1/\rho)}{\log (N/\rho)} \to +\infty \,.
\]
Note that this choice is possible because \(\frac{\log(1/\rho)}{\log (N/\rho)} \log N \to +\infty\) in the regime \(\rho\to 0\), \(N\to\infty\).

Now, note that, by \eqref{eq:moments-T}, we have
\[
  \mathrm E_{1/N}[T_1]
  = -\frac{\Rhat'(1/N)}{\Rhat(1/N)}
  = \frac{N(1+o(1))}{\log N}.
\]
Therefore, by Markov's inequality we obtain that
\[
  \inf_{k<L_N}\mathrm P_{1/N}(\tau_k\le N)
  =  \mathrm P_{1/N}(\tau_{L_N}\le N)
  \geq 1 - (1+o(1)) \frac{L_N}{\log N}
  = 1-o(1) \,,
\]
using that \(L_N/\log N\to 0\) for the last identity.

Coming back to the inequality~\eqref{eq:sum-restriction-1}, since \(\Rhat(1/N)<\Rhat(\rho/N)\), we have that
\[
  \mathcal{V}_N^{\beta}
  \geq (1-o(1)) \frac{1}{\Rhat(\rho/N)} \sum_{k=0}^{L_N} \Bigl( \frac{\Rhat(1/N)}{\Rhat(\rho/N)}\Bigr)^k
  = (1-o(1))  \frac{1 - \Bigl( \frac{\Rhat(1/N)}{\Rhat(\rho/N)}\Bigr)^{L_N+1}}{\Rhat(\rho/N) -\Rhat(1/N)} \,.
\]
Now, by \eqref{eq:Rhat-asymptotic} we have that
\[
  \Rhat(\rho/N)-\Rhat(1/N)
  = \frac{\log(1/\rho)+o(1)}{\pi}\,.
\]
On the other hand, using this identity together with \(\pi\Rhat(\rho/N) = (1+o(1)) \log(N/\rho)\), we also have
\[
  \frac{\Rhat(1/N)}{\Rhat(\rho/N)}
  = 1- \frac{\Rhat(\rho/N)-\Rhat(1/N)}{\Rhat(\rho/N)}
  = 1 - (1+o(1)) \frac{\log(1/\rho)}{\log(N/\rho)} \,.
\]
Since \(L_N \frac{\log(1/\rho)}{\log (N/\rho)} \to +\infty\), this shows that \(\bigl(\frac{\Rhat(1/N)}{\Rhat(\rho/N)}\bigr)^{L_N+1} \to 0\).

All together, we end up with
\[
  \mathcal{V}_N^{\beta}
  \ge (1+o(1)) \frac{\pi}{\log(1/\rho)}
  = \pi\, \mathcal{V}((1+o(1)) \rho)\,,
\]
using again that \(\mathcal{V}(T) \sim \frac{1}{\log(1/T)}\) as \(T\to0\), see~\eqref{eq:asymp-V}.
Recalling that \(\rho=(1+o(1))\e^{\th(N,\beta)-\gamma}\), see~\eqref{eq:rho-theta}, this gives the desired lower bound in the subcritical regime \(\th(N,\beta) \to -\infty\).
\qed

\subsection*{Acknowledgements}
We thank Francesco Caravenna and Rongfeng Sun for numerous helpful discussions.
We also thank Clément Cosco, Shuta Nakajima and Ofer Zeitouni for interesting discussions about the inhomogeneous scales of the coarse-graining procedure. 
Q.B. acknowledges the support of Institut Universitaire de France and ANR Local (ANR-22-CE40-0012-02).
N.T. acknowledges the support of INdAM/GNAMPA.
The work of N.Z. was partially supported by EPSRC grant APP57502. 
N.Z. also acknowledges the support and hospitality of the National Centre for Theoretical Sciences, Taipei, where parts of this work were completed.

\printbibliography

\end{document}